\documentclass[12pt]{amsart}
\usepackage{textcomp}
\usepackage{amsmath,amsfonts,amssymb,amsthm, multicol}
\usepackage{anysize}
\marginsize{2cm}{2cm}{2cm}{2cm}
\usepackage{graphicx}
\usepackage{color}
\usepackage[pdftex]{hyperref}
\usepackage{comment}
\usepackage{tikz}
    \usetikzlibrary{cd} 
    \usetikzlibrary{calc} 

\usepackage{enumerate}
\usepackage{caption}
\usepackage{ gensymb }

\usepackage{mathtools}

\def\a{\mathfrak{a}}

\def\g{\mathfrak{g}}
\def\h{\mathfrak{h}}
\def\k{\mathfrak{k}}

\def\n{\mathfrak{n}}

\def\t{\mathfrak{t}}
\def\u{\mathfrak{u}}

\def\z{\mathfrak{z}}

\def\R{\mathbb{R}}
\def\Q{\mathbb{Q}}
\def\Z{\mathbb{Z}}
\def\N{\mathbb{N}}

\def\e{\operatorname{e}}

\def\ad{\operatorname{ad}}
\def\tr{\operatorname{tr}}

\def\alt{\raise1pt\hbox{$\bigwedge$}}

\def\im{\operatorname{im}}
\def\Id{\operatorname{Id}}
\def\GL{\operatorname{GL}}
\def\Ad{\operatorname{Ad}}

\def\mid{\, \vert \,} 
\def\Span{\operatorname{span}}

\def\diag{\operatorname{diag}}
  
\newcommand\aff{\mathfrak{aff}}

\theoremstyle{plain}
\newtheorem{theorem}{\bf Theorem}[section]
\newtheorem{corollary}[theorem]{\bf Corollary}
\newtheorem{proposition}[theorem]{\bf Proposition}
\newtheorem{lemma}[theorem]{\bf Lemma}

\theoremstyle{definition}
\newtheorem{definition}[theorem]{\bf Definition}
\newtheorem{example}[theorem]{\bf Example}

\theoremstyle{remark}
\newtheorem{remark}[theorem]{\bf Remark}

\usepackage[pageref]{backref}
\renewcommand*{\backrefalt}[4]{%
  \ifcase #1 %
    (Not cited.)%
  \or
    (Cited on page~\hyperlink{page.#2}{#2}.)%
  \else
    (Cited on pages~\hyperlink{page.#2}{#2}.)%
  \fi
}

\newcommand{\cabe}[1]{\textcolor{red}{#1}}

\title{1-Lefschetz contact solvmanifolds}

\author{Adrián Andrada}
\email{adrian.andrada@unc.edu.ar}
\author{Agustín Garrone}
\email{agustin.garrone@unc.edu.ar} 

\date{}
\address{FAMAF, Universidad Nacional de C\'ordoba and CIEM-CONICET, Av. Medina Allende s/n, Ciudad Universitaria, X5000HUA C\'ordoba, Argentina}

\thanks{This work was partially supported by CONICET, SECyT-UNC and ANPCyT (Argentina)}

\subjclass[2020]{53C30, 53C15, 22E60, 22E25, 22E40}
\keywords{}

\begin{document}

\begin{abstract}      
    We study the contact $1$-Lefschetz condition on compact contact solvmanifolds with an invariant contact form, as introduced by B.\ Cappelletti-Montano, A.\ De Nicola and I.\ Yudin. We prove that the $1$-Lefschetz condition on Lie algebras is preserved via $1$-dimensional central extensions by a symplectic cocycle, thereby establishing that a unimodular symplectic Lie algebra $(\mathfrak{h}, \omega)$ is $1$-Lefschetz if and only if its contactization $(\mathfrak{g}, \eta)$ is $1$-Lefschetz. We achieve this equivalence by showing an explicit relation between the relevant cohomology degrees of $\mathfrak{h}$ and $\mathfrak{g}$, and also between the commutators $[\mathfrak{h},\mathfrak{h}]$ and $[\mathfrak{g},\mathfrak{g}]$. By specializing to the nilpotent setting, we prove that $1$-Lefschetz contact nilmanifolds equipped with an invariant contact form are quotients of a Heisenberg group by a lattice, and deduce that there are many examples of compact $K$-contact solvmanifolds not admitting compatible Sasakian structures. Lastly, we construct new examples of completely solvable $1$-Lefschetz solvmanifolds, some having the $2$-Lefschetz property and some failing it.   
\end{abstract}

\maketitle

\section{Introduction}

\indent On a compact symplectic manifold $(N^{2n},\omega)$, whose algebra of differential forms is denoted by $\Omega(N)$, the $s$-Lefschetz condition can be stated as the fact that the \textit{Lefschetz operators} 
$L^k \colon H_{dR}^{n-k}(N) \to H_{dR}^{n+k}(N)$, induced on de Rham cohomology from the linear isomorphisms $\Omega^{n-k}(N) \to \Omega^{n+k}(N)$ given by $\alpha \mapsto \omega^k \wedge \alpha$, are bijective for every $0 \leq k \leq s$, where $0 \leq s \leq n$. The importance of this property lies primarily in the fact that every compact Kähler manifold is $n$-Lefschetz, a feature usually called the \textit{hard-Lefschetz condition}, and moreover it is closely related to other Kähler-like cohomological properties. This is outlined briefly in Section \ref{section: the lefschetz condition}. 

\indent A natural class of manifolds in which to study the Lefschetz condition is given by solvmanifolds, since they have been historically a source of examples and counterexamples in differential geometry. Recall that $N$ is a \textit{solvmanifold} when it can be written as a compact quotient $\Gamma \backslash G$ of some simply connected solvable Lie group $G$ by a discrete co-compact subgroup $\Gamma$ of $G$, called a \textit{lattice}. A brief review of solvmanifolds and some of its properties can be found in Section \ref{section: solvmanifolds}. If the cohomology of $\Gamma \backslash G$ is isomorphic to the invariant one (see Proposition \ref{prop: sometimes it is an isomorphism} for when this may happen), questions concerning the Lefschetz condition can be posed and answered at the infinitesimal level. In this regard, Benson and Gordon have shown that $1$-Lefschetz nilmanifolds (i.e., solvmanifolds with nilpotent $G$) are tori in \cite[proof of Theorem A]{BG1}, irrespective of whether the symplectic form is invariant or not; moreover, they also find a characterization of unimodular symplectic Lie algebras (which are solvable due to \cite[Theorem 11]{Chu}) that are $1$-Lefschetz in \cite[proof of Theorem 2, Remarks in Section 2]{BG2}, thus generalizing their previous work. 

\indent It is then natural to ask for natural extensions of these ideas and results to compact contact manifolds. We review the basics of contact geometry in Section \ref{section: the lefschetz condition}. However, complications quickly arise, since there are no natural candidates for Lefschetz operators to begin with. In this article, we are concerned with the contact Lefschetz condition as first defined in \cite{Cagliari 2}. In there, for a given $0 \leq k \leq n$, the \textit{Lefschetz relation in degree $k$} for a contact manifold $(M^{2n+1}, \eta)$, with Reeb vector field $\xi$, is defined as the following subset of $H_{dR}^k(M) \times H_{dR}^{2n+1-k}(M)$:
\begin{align} \label{eq: lefschetz relation}
    \mathcal{R}_{ \mathrm{Lef}_k } = \{ ( [\beta], [\epsilon_{\eta} L^{n-k} (\beta)] ) \mid \beta \in \Omega^k(M), \; d \beta = 0, \; \iota_{\xi} \beta = 0, \; L^{n-k+1} \beta = 0 \}.
\end{align}
\noindent Here, $\epsilon_{\eta}(\gamma) := \eta \wedge \gamma$ for all $\gamma \in \Omega(M)$. We then say that $(M^{2n+1},\eta)$ is \textit{$s$-Lefschetz} if for all $0 \leq k \leq s$ the relation $\mathcal{R}_{ \mathrm{Lef}_k }$ in \eqref{eq: lefschetz relation} is the graph of an isomorphism $\mathrm{Lef}^{n-k}\colon H^k_{dR}(M) \to H^{2n+1-k}_{dR}(M)$, where $0 \leq s \leq n$. We point out that this definition involves several subtleties worth discussing (see the remarks below Definition \ref{def: contact lefschetz condition}). As expected from the situation in the symplectic case, compact Sasakian manifolds are hard-Lefschetz, as proven in \cite[Theorem 3.6, Theorem 4.5]{Cagliari 2}.

\indent Unlike the symplectic case, the contact Lefschetz condition in the solvmanifold setting has not been studied as intensely. There is an analogue to Benson and Gordon's result for nilmanifolds in the contact setting, proven independently in \cite[Theorem 1.1]{Cagliari 3} and in \cite[Theorem 8.2]{Kasuya3}, stating that the only Sasakian nilmanifolds (not necessarily with invariant Sasakian structure) are quotients of a Heisenberg group. However, there appears to be no result characterizing contact $1$-Lefschetz solvmanifolds, not even in the nilpotent case. One of the goals of this article is to address this problem. 

\indent Our approach is to work primarily at the Lie algebra level, and then pass to compact quotients whenever lattices exist. A well-known necessary condition for this is unimodularity, and thus all of our Lie algebras are assumed to be unimodular. Moreover, we also assume our Lie algebras to have nontrivial center, the reason for this being that all nilpotent Lie algebras are as such, and we aim for a result for a class of solvmanifolds including nilmanifolds. As we recall in Section \ref{section: solvmanifolds and the contact Lefschetz condition.}, contact unimodular Lie algebras $(\g, \eta)$ with nontrivial center are in bijective correspondence  with unimodular symplectic Lie algebras $(\h, \omega)$ via \textit{contactization}, i.e. one-dimensional central extensions by a symplectic cocycle. In particular, both $\h$ and $\g$ turn out to be solvable, due to \cite[Theorem 11]{Chu}. One of our main results is the fact that this correspondence preserves the $1$-Lefschetz condition, a result that can be stated with precision as follows. 
\begin{theorem} \label{thm: main}
    Let $(\g, \eta)$ be a contact unimodular Lie algebra with nontrivial center, thereby arising from contactization of a unimodular symplectic Lie algebra $(\h, \omega)$. Then $(\g,\eta)$ is $1$-Lefschetz if and only if $(\h,\omega)$ is $1$-Lefschetz. 
\end{theorem} 
    \indent Theorem \ref{thm: main} is a partial analogue of the Benson–Gordon result for unimodular symplectic Lie algebras in \cite[proof of Theorem 2, Remarks in Section 2]{BG2}, in the sense that it characterizes contact unimodular Lie algebras satisfying the $1$-Lefschetz condition in terms of Lie-theoretic data.  

    \indent The proof of Theorem \ref{thm: main} takes up about all Section \ref{section: cohomology remarks} and most of Section \ref{section: the 1-Lefschetz condition for contact Lie algebras arising from contactization}. An analogous result for contact solvmanifolds is also valid, provided that the contact form is invariant and their cohomology can be computed from invariant forms (see Remark \ref{obs: main, i guess}), for example, in the nilpotent and completely solvable cases. We note that Theorem \ref{thm: main} could be deduced following the proof of \cite[Theorem 8.1]{Lin}; however, our proof is different and actually elementary. 
        
    \indent We obtain the following result as a consequence of applying Theorem \ref{thm: main} in the nilpotent setting.
\begin{theorem} \label{thm: main applications}
    \phantom{.}
\begin{enumerate} [\rm (i)]
    \item A nilmanifold endowed with an invariant contact form is $1$-Lefschetz if and only if it is a Heisenberg nilmanifold. 
    \item Any non-Heisenberg nilmanifold endowed with an invariant contact form admits a compatible $K$-contact metric but does not admit a compatible (not necessarily invariant) Sasakian structure.  
\end{enumerate} 
\end{theorem} 
    \indent These assertions are proven in Section \ref{section: some applications}, and are stated separately in Theorem \ref{thm: odd benson-gordon} and in Corollary \ref{cor: cagliari}, respectively. Notice that Theorem \ref{thm: main applications}(i) can be regarded as the contact analogue of the classical Benson and Gordon's result characterizing $1$-Lefschetz symplectic nilmanifolds, found in \cite[proof of Theorem A]{BG1}, albeit with the further assumption that the contact form under consideration is invariant. It is unclear as of now if the same result holds in the non-invariant case. Theorem \ref{thm: main applications}(ii) helps place the earlier isolated examples of \cite{Cagliari 1} into a broader picture.
    
    \indent In the course of our proof of Theorem \ref{thm: main}, we relate the commutators $\n_{\g} = [\g,\g]_{\g}$ and $\n_{\h} = [\h,\h]_{\h}$ of $\g$ and $\h$, respectively. The precise relation is given in Proposition \ref{prop: xi belongs in the commutator}. If further either $(\g, \eta)$ or $(\h, \omega)$ is $1$-Lefschetz, and so is the other due to Theorem \ref{thm: main}, this relation is enhanced. See Corollary \ref{cor: boring benson-gordon} for details. 
    
    \indent Lastly, in Section \ref{section: lattices} we construct examples of contact $1$-Lefschetz solvmanifolds. This is one of the main contributions of this article, since, to the best of our knowledge, examples of 1-contact manifolds are scarce. Our examples come in three kinds: the first ones (in Section \ref{section: semisimple case}) are also $2$-Lefschetz, and the second and third ones (in Section \ref{section: nonsemisimple case} and Section \ref{section: BG case}, respectively) are never $2$-Lefschetz. We obtain them by showing that the contactization of some of the Lie algebras the authors studied in \cite{AG1} and \cite{AG2}, and also one appearing in \cite[Example 3]{BG2}, admit lattices; moreover, all examples are of completely solvable type, and thus the Lefschetz condition on the corresponding solvmanifolds follows from the properties of their Lie algebras. 

\section{Preliminaries}
\subsection{The Lefschetz condition} \label{section: the lefschetz condition}

    \indent A (co-oriented) \textit{contact manifold} $M^{2n+1}$ is an odd-dimensional smooth manifold endowed with a maximally non-integrable codimension-one distribution $\mathcal{D}$, known as a \textit{contact structure}, given by $\mathcal{D} = \ker \eta$, where $\eta \in \Omega^1(M)$ is a \textit{contact form} on $M^{2n+1}$; that is, a $1$-form such that
\begin{align*}
    \eta \wedge (d\eta)^n \neq 0
\end{align*}
    \noindent everywhere on $M$. We  refer to the pair $(M^{2n+1}, \eta)$ as a \textit{contact manifold}. The non-degeneracy condition means that $\eta \wedge (d\eta)^n$ defines a volume form on $M^{2n+1}$, hence every contact manifold is orientable. Moreover, $(\mathcal{D}, d \eta)$ is a symplectic vector bundle. 

    \indent Every contact manifold $(M^{2n+1}, \eta)$ possesses a distinguished global vector field $\xi \in \mathfrak{X}(M)$, called the \textit{Reeb vector field}, which is determined uniquely from the conditions
\begin{align*}
    \iota_{\xi} d\eta = 0, \qquad \eta(\xi) = 1.
\end{align*} 
    \noindent The Reeb vector field $\xi$ defines the characteristic foliation with one-dimensional leaves, giving rise to a canonical splitting of the tangent bundle $TM$ of $M^{2n+1}$,
\begin{align*}
    TM = \mathcal{D} \oplus \mathcal{L},
\end{align*}
\noindent with $\mathcal{L}$ being the trivial line bundle generated by $\xi$. 

    \indent Let $(N^{2n}, \omega)$ be a symplectic manifold. For each $0 \leq k \leq n$, there are linear bijections,  
\begin{align*}
    L^{n-k}\colon \Omega^k(N) \to \Omega^{2n-k}(N), \quad L^{n-k}(\alpha) = \omega^{n-k} \wedge \alpha,
\end{align*}    
    \noindent inducing corresponding operators in de Rham cohomology, 
\begin{align*}
    L^{n-k}\colon H^k_{dR}(N) \to H^{2n-k}_{dR}(N), \quad L^{n-k}( [\alpha] ) = [\omega^{n-k} \wedge \alpha],
\end{align*}
    \noindent known as \textit{Lefschetz operators}. Both the operators at forms-level and at cohomology-level are referred to by the same name, since there is no risk of confusion. Usually, the Lefschetz operators are not bijective in cohomology, and a standard question in the literature is whether or not they are all bijective for a specific symplectic manifold or family of manifolds. As first noted in \cite{FMU}, keeping track of where this condition fails for the first time is also of interest. 
\begin{definition} \label{def: symplectic lefschetz condition}
    A symplectic manifold $(N^{2n},\omega)$ is said to be \textit{$s$-Lefschetz}, where $0 \leq s \leq n$ is given, if $L^{n-k}$ is bijective for all $0 \leq k \leq s$. 
\end{definition}
    \indent Occasionally, instead of calling a symplectic manifold \textit{$s$-Lefschetz}, we speak of it as \textit{having the $s$-Lefschetz property} or \textit{satisfying the $s$-Lefschetz condition}. Notice that $L^{n-k}$ is always a bijection when $k = n$, so it is only important to consider $s \leq n-1$. As mentioned above, in the particular case when $s = n-1$, $(N^{2n}, \omega)$ is called \textit{hard-Lefschetz}. In \cite{FMU}, a symplectic manifold that is $s$-Lefschetz but not hard-Lefschetz is called \textit{weak-Lefschetz}, but we do not use that terminology in this article.  
    
    \indent It is well known that, when $N$ is compact, the Lefschetz condition on $N$ is closely tied to the existence of a symplectic Hodge theory, a symplectic analogue of the classical Hodge theory in the Riemannian setting. See \cite[Section 2]{AG2} for a quick review of these facts, \cite{Tseng-Yau 1} and \cite{Tseng-Yau 2} for further development of that connection, and \cite[Section 2]{FMU} for the corresponding extension to the weak case. It is a classical result that every compact Kähler manifold satisfies the hard-Lefschetz condition. Historically, this fact motivated the introduction of the Lefschetz condition in symplectic geometry, as a way to capture when a manifold is “cohomologically Kähler’’ in a purely symplectic sense.

    \indent There are several analogues of the Lefschetz condition in the contact setting. In \cite{EKA}, a variant using $\xi$-basic cohomology instead of the de Rham one is proposed. In this article, we are concerned only with the one put forward in \cite{Cagliari 2}. In there, for a given $0 \leq k \leq n$, the \textit{Lefschetz relation in degree $k$} for a contact manifold $(M^{2n+1}, \eta)$ is defined as the following subset of $H_{dR}^k(M) \times H_{dR}^{2n+1-k}(M)$:
\begin{align*}
    \mathcal{R}_{ \mathrm{Lef}_k } = \{ ( [\beta], [\epsilon_{\eta} L^{n-k} (\beta)] ) \mid \beta \in \Omega^k(M), \; d \beta = 0, \; \iota_{\xi} \beta = 0, \; L^{n-k+1} \beta = 0 \}.
\end{align*}
\noindent Here, $\epsilon_{\eta}(\gamma) := \eta \wedge \gamma$ for all $\gamma \in \Omega(M)$. Forms $\beta \in \Omega^k(M)$ satisfying $i_{\xi} \beta = 0$ are called \textit{$\xi$-horizontal}, and forms satisfying $L^{n-k+1} \beta = 0$ are called \textit{primitive}. Notice that, since we are dealing with closed forms $\beta$, the fact that they are $\xi$-horizontal implies that they are also $\xi$-basic, meaning that they are also in the kernel of the Lie derivative $\mathcal{L}_{\xi}$ with respect to $\xi$.  

\begin{definition} \label{def: contact lefschetz condition}
    A contact manifold $(M^{2n+1},\eta)$ is said to be \textit{$s$-Lefschetz}, where $0 \leq s \leq n$ is given, if for all $0 \leq k \leq s$ the relation $\mathcal{R}_{ \mathrm{Lef}_k }$ in equation \eqref{eq: lefschetz relation} is the graph of an isomorphism $\mathrm{Lef}^{n-k}\colon H^k_{dR}(M) \to H^{2n+1-k}_{dR}(M)$. 
\end{definition}

    \indent When confusion might arise between Definitions \ref{def: symplectic lefschetz condition} and \ref{def: contact lefschetz condition}, we speak about the \textit{symplectic Lefschetz condition} and the \textit{contact Lefschetz condition}, respectively. 

    \indent Definition \ref{def: contact lefschetz condition} involves several subtleties worth discussing. Notice that it does \underline{not} assert that the linear operators $\mathrm{Lef}^{n-k}( [\beta] ) := [\epsilon_{\eta} L^{n-k}(\beta)]$, defined on the subset of $H^k_{dR}(M)$ of cohomology classes $[\beta]$ admitting $\xi$-horizontal and primitive representatives $\beta$, are bijective. In fact this is half the picture: it also requires that said $\mathrm{Lef}^{n-k}$ has domain \textit{exactly} $H^k_{dR}(M)$, meaning that \textit{all} cohomology classes in $H^k_{dR}(M)$ have $\xi$-horizontal and primitive representatives. And, of course, that this holds for all $0 \leq k \leq s$. 
    
    \indent Just as contact geometry serves as an odd-dimensional counterpart of symplectic geometry, \textit{metric contact structures} are odd-dimensional analogues of almost Hermitian geometry. It is well known that any contact manifold $(M^{2n+1}, \eta)$ has a Riemannian metric $g$ and a $(1,1)$-tensor field $\Phi$ subject to the following relations:
\begin{align*}
    \eta = \iota_{\xi} g, \quad d \eta = 2 g(\cdot, \Phi \, \cdot), \quad \Phi^2 = - \Id + \eta \otimes \xi.
\end{align*}
    \noindent Here, $\Id:TM \to TM$ is the identity mapping. All these imply at once that $\Phi \xi = 0$ and $\eta \circ \Phi = 0$, as well as the compatibility condition
\begin{align*}
    g(\Phi X, \Phi Y) = g(X,Y) - \eta(X) \eta(Y) \text{ for all $X$, $Y \in \mathfrak{X}(M)$}. 
\end{align*}
    \noindent We refer to \cite[Section 4]{Blair} for details. A triple $(\eta, g, \Phi)$ satisfying all of the above is called a \textit{metric contact structure}.  

    \indent There are also analogues of almost-Kähler and Kähler structures in this metric contact setting, also relevant to our purposes: they are \textit{$K$-contact} and \textit{Sasakian} manifolds, respectively. To define them, consider a metric contact structure $(\eta, g, \Phi)$ on $M$. Denote by $\nabla$ the Levi-Civita connection associated to the metric $g$ on $M$; also, let $N_{\Phi}$ be the Nijenhuis tensor associated to $\Phi$, which is defined as
\begin{align*}
    N_{\Phi}(X,Y) = \Phi^2[X,Y] +[\Phi X, \Phi Y] - \Phi[\Phi X, Y] - \Phi[X, \Phi Y] \text{ for all $X$, $Y \in \mathfrak{X}(M)$}.
\end{align*}     
\begin{proposition} \label{prop: K-contact equivalences}
    The following conditions are equivalent:
\begin{enumerate} [\rm (i)]
    \item $\mathcal{L}_{\xi} g = 0$. 
    \item $\mathcal{L}_{\xi} \Phi = 0$.  
    \item $\Phi X = - \nabla_X \xi$ for all $X \in \mathfrak{X}(M)$. 
\end{enumerate}    
\end{proposition}
\begin{proof}
    See \cite[Theorem 6.2]{Blair} and \cite[Lemma 6.2]{Blair}. 
\end{proof}
    \indent A contact metric structure $(\eta, g, \Phi)$ on $M$ is said to be \textit{$K$-contact} if any (equivalently, all) of the conditions in Proposition \ref{prop: K-contact equivalences} hold. A contact manifold $(M,\eta)$ is \textit{$K$-contact} if it admits a compatible $K$-contact metric structure. 
\begin{proposition} \label{prop: Sasaki equivalences}
    The following conditions are equivalent:
\begin{enumerate} [\rm (i)]
    \item $N_{\Phi}(X,Y) = - d \eta(X,Y) \xi$ for all $X$, $Y \in \mathfrak{X}(M)$.  
    \item $(\nabla_X \Phi)Y = g(X,Y) \xi - \eta(Y) X$ for all $X$, $Y \in \mathfrak{X}(M)$.   
\end{enumerate}    
\end{proposition}
\begin{proof} 
    See \cite[Theorem 6.3]{Blair}.
\end{proof}
    \indent A contact metric structure $(\eta, g, \Phi)$ on $M$ is said to be \textit{Sasakian} if any (equivalently, all) of the conditions in Proposition \ref{prop: Sasaki equivalences} hold. A contact manifold $(M,\eta)$ is \textit{Sasakian} if it admits a compatible Sasakian metric structure. Notably, all Sasakian manifolds are $K$-contact (see \cite[Corollary 6.3]{Blair}). While the converse is true in dimension 3, it is not true in general (see \cite[Chapter 6, Section 7]{Blair} for examples, although we discuss some examples below and in other sections). 

    \indent Just as compact Kähler manifolds, compact Sasakian manifolds exhibit the hard-Lefschetz property. 
\begin{theorem} \cite[Theorem 3.6, Theorem 4.5]{Cagliari 2} \label{thm: compact sasakian are hard-Lefschetz}
    Compact Sasakian manifolds are hard-Lefschetz.
\end{theorem}
    \indent Theorem \ref{thm: compact sasakian are hard-Lefschetz} is established proving that the operators $\alpha \mapsto \eta \wedge (d \eta)^{n-k} \wedge \alpha$ define isomorphisms between the spaces of (metrically) harmonic forms $\Omega_{\Delta}^k$ and $\Omega_{\Delta}^{2n+1-k}$, and thus between $H^k_{dR}(M)$ and $H^{2n+1-k}_{dR}(M)$, for all $0 \leq k \leq n$. Crucially, it is also established that such isomorphisms are independent of the choice of a compatible Sasakian metric, and thus are bona fide contact invariants. It is in this context that the rather technical Definition \ref{def: contact lefschetz condition} arises, and with it some cohomological obstructions for a contact manifold to admit compatible Sasakian structures: For example, similarly to the symplectic case, the $k$-th Betti number of a contact Lefschetz compact manifold $(M^{2n+1}, \eta)$ is an even number if either $k$ is odd and $k \leq n$ or $k$ is even and $k \geq n+1$ (see \cite[Theorem 5.2]{Cagliari 2}). For compact Sasakian manifolds, this result was known through other methods (see \cite[Theorem 4.4]{Fujitani}). 
    
    \indent In \cite[Theorem 6.3]{Lin}, it is shown that the $s$-Lefschetz condition for $K$-contact manifolds is equivalent to a similar property for operators on basic cohomology (described in \cite{EKA}); in the same article, the relation between the Lefschetz condition and odd-dimensional analogues of Hodge theory is studied.  
    
    \indent No result for $K$-contact compact manifolds analogous to Theorem \ref{thm: compact sasakian are hard-Lefschetz} can exist, as counterexamples are known. A noteworthy family of examples of this kind is provided in \cite{Cagliari 4}: they are $5$-dimensional, $K$-contact, hard-Lefschetz, formal in the sense of Sullivan, and of Tievsky type (a condition on the minimal model known to be possessed by all Sasakian manifolds), and they arise as $S^1$-bundles over a $4$-dimensional compact symplectic solvmanifold of completely solvable type. See, in particular, \cite[Theorem 4.1]{Cagliari 4}. We review solvmanifolds in the next two sections. 

\subsection{Solvmanifolds} \label{section: solvmanifolds}

\indent Throughout the article, let $G$ denote a connected real Lie group with Lie algebra $\g$. 
\begin{definition}
    A \textit{solvmanifold} is a compact quotient $\Gamma \backslash G$, where $G$ is simply connected and solvable, and $\Gamma$ is a discrete subgroup of $G$. Such a co-compact discrete subgroup $\Gamma$ is called a \textit{lattice} of $G$. If $G$ is nilpotent then $\Gamma \backslash G$ is called a \textit{nilmanifold}. 
\end{definition}
    \indent Special classes of solvmanifolds relevant in what follows, apart from nilmanifolds, are those arising from completely solvable Lie groups $G$; we call them  \textit{completely solvable solvmanifolds}. Recall that a connected solvable Lie group $G$ is \textit{completely solvable} if all the adjoint operators $\ad_x: \g \to \g$, with $x\in \g$, have only real eigenvalues.

    \indent As every connected and simply connected solvable Lie group is diffeomorphic to $\R^n$, the usual argument involving the long exact sequence associated to a fibration implies that every solvmanifold $\Gamma \backslash G$ is aspherical, meaning that $\pi_n(\Gamma\backslash G) = 0$ for all $n > 1$, as well as $\pi_1(\Gamma \backslash G) = \Gamma$. Moreover, a classical result due to Mostow shows that solvmanifolds are \textit{rigid}, meaning that they are determined up to diffeomorphism by their fundamental groups (see \cite[Theorem A]{Mostow 1}).

    \indent Solvmanifolds are prominent in the study of geometric structures, as they are usually sources of interesting examples and counterexamples of various kinds. Part of their usefulness comes from the fact that they are algebraically well-behaved. For instance, knowledge of the cohomology of the Lie algebra $\g$ associated to a solvmanifold $\Gamma \backslash G$ gives considerable information about the de Rham cohomology of $\Gamma \backslash G$. 
\begin{proposition} \cite[Theorem 8.1]{Mostow 2} \label{prop: there is always an injection}
    There is an injection $H^*(\g) \hookrightarrow H^*_{dR}(\Gamma\backslash G)$ induced in cohomology by the natural inclusion $\alt^* \g^* \hookrightarrow \Omega^*(\Gamma\backslash G)$. 
\end{proposition}  
    \indent It is remarkable that Proposition \ref{prop: there is always an injection} holds in full generality, not needing to impose further conditions on either $G$ or $\Gamma$. In addition to the original article of Mostow \cite{Mostow 2}, the reader can find a nice proof of Proposition \ref{prop: there is always an injection} in \cite[Theorem 7.26 and Remark 7.30]{Raghunathan}. As an easy consequence, the first Betti number $b_1$ of a solvmanifold $\Gamma \backslash G$ is strictly greater than zero, since 
\begin{align*}
    b_1 = \dim H^1(\Gamma \backslash G) \geq \dim H^1(\g) = \dim \g/[\g,\g] \geq 1.
\end{align*} 
    \noindent Here, we are using that $\g \neq [\g,\g]$ for solvable $\g$ and that $H^1(\g) \cong \left( \g/[\g,\g] \right)^*$ to derive $b_1 \geq 1$.
    
    \indent In many common situations, the injection in Proposition \ref{prop: there is always an injection} is in fact bijective. 
\begin{proposition} \label{prop: sometimes it is an isomorphism}
    The natural inclusion $\alt^* \g^* \hookrightarrow \Omega^*(\Gamma\backslash G)$ induces an isomorphism in cohomology in either any of the following cases:
\begin{enumerate} [\rm (i)]
    \item $G$ is a nilpotent Lie group.
    \item $G$ is a completely solvable Lie group.
    \item $\Ad_G(G)$ and $\Ad_G(\Gamma)$ have the same Zariski-closures in $\GL(\g)$. 
\end{enumerate}
\end{proposition}
    \indent As stated, Proposition \ref{prop: sometimes it is an isomorphism} is a collection of well-known independent results: (i) is credited to Nomizu \cite[Theorem 1]{Nomizu}, (ii) is attributed to Hattori \cite[Theorem 4.1]{Hattori}, and (iii) is due to Mostow \cite[Theorem 8.1]{Mostow 2}. In fact, the condition stated in (iii) has come to be known as the \textit{Mostow condition}. Nice alternative proofs of (i) and (iii) are given in \cite[Corollary 7.28 and Corollary 7.29]{Raghunathan}. Notice that (i) and (ii) are actually particular cases of (iii), although not obviously so. 

    \indent A nilmanifold $M = \Gamma \backslash N$ is said to be a \textit{Heisenberg nilmanifold} if $N$ is isomorphic to the Heisenberg Lie group $H$ of the corresponding dimension. Heisenberg nilmanifolds are featured prominently in this article, partly because they admit Sasakian structures; moreover, they are the only nilmanifolds to do so (see Theorem \ref{thm: sasakian nilmanifolds} below). One of our main results is a partial generalization of this fact (see Theorem \ref{thm: odd benson-gordon}), and holds in part due to a result of Malcev we now recall.  
\begin{proposition} \cite{Malcev} \label{prop: Malcev rigidity}
    If $\Gamma_1$ and $\Gamma_2$ are lattices in simply connected nilpotent Lie groups $N_1$ and $N_2$ respectively, then every isomorphism $f\colon \Gamma_1 \to \Gamma_2$ extends uniquely to a Lie group isomorphism $F\colon N_1 \to N_2$. In particular, $\Gamma_1 \backslash N_1$ is diffeomorphic to $\Gamma_2 \backslash N_2$.
\end{proposition}
    \indent Proposition \ref{prop: Malcev rigidity} is colloquially termed \textit{Malcev's rigidity theorem}. See \cite[Corollary 2 of Theorem 2.10]{Raghunathan} for a proof of Theorem \ref{prop: Malcev rigidity}. There is a well-known generalization to completely solvable solvmanifolds, due to Saito \cite[Theorem 5]{Saito1}. Note that both these results can be seen as enhancements of Mostow's result on the rigidity of fundamental groups of solvmanifolds mentioned above. Our main use of Proposition \ref{prop: Malcev rigidity} is the next remark. 
\begin{remark} \label{obs: heisenberg nilmanifolds}
     Proposition \ref{prop: Malcev rigidity} ensures that any nilmanifold diffeomorphic to a Heisenberg nilmanifold is in fact a Heisenberg nilmanifold. For if $\Gamma_N \backslash N$ and $\Gamma_H \backslash H$ are diffeomorphic then their fundamental groups $\Gamma_N$ and $\Gamma_H$ are isomorphic as abstract groups, and we can extend that isomorphism to a Lie group isomorphism between $N$ and $H$. 
\end{remark}
    \indent Recall that lattices of Heisenberg groups have been classified in \cite[Theorem 2.4]{GW}. Endeavors like this for other solvable Lie groups are close to impossible unless strong restrictions are imposed. The question of whether lattices exist for a given solvable Lie group is also quite daunting. It is known that, for any fixed dimension,  only countably many non-isomorphic simply connected solvable Lie groups admit lattices (see \cite[Theorem 4]{Milovanov}); the same result is true even for general simply connected Lie groups (see \cite[Proposition 8.7]{Witte}). For a simply connected nilpotent Lie group $N$, a result of Malcev \cite{Malcev} ensures that the existence of lattices is equivalent to the Lie algebra $\n$ of $N$ having a rational basis (see \cite[Theorem 2.12]{Raghunathan} for a proof): Saying that the basis $\mathcal{B}$ of $\n$ is \textit{rational} means that all structure constants of the Lie bracket of $\n$ with respect to $\mathcal{B}$ are rational numbers. Also, there are known criteria for general solvable Lie groups (see \cite[Chapter 2, Section 3.7]{VGS}), but they are cumbersome and impractical unless very specific subclasses of solvable Lie groups are considered (see, for example, Proposition \ref{prop: yamada} below). A rather weak necessary condition for their existence that is most relevant for the concerns of this article is recalled next. 
\begin{proposition} \label{prop: milnor}
    If $G$ has lattices then it is unimodular; i.e., $\tr(\ad_x) = 0$ for all $x \in \g$.  
\end{proposition}
    \indent Proposition \ref{prop: milnor} is true for general connected Lie groups. There are proofs of this classical result in \cite[Lemma 6.2]{Milnor} and in \cite[Remark 1.9]{Raghunathan}. 

    \indent The main results of the present article are established at the Lie algebra level. We then argue that, for some choice of Lie algebras, the corresponding simply connected Lie groups have lattices. Since we are interested in cohomological properties, Propositions \ref{prop: there is always an injection} and \ref{prop: sometimes it is an isomorphism} are used to derive facts on the corresponding solvmanifolds. We recall some important results in this context in the following two sections. 

\subsection{Solvmanifolds and the symplectic Lefschetz condition}

    \indent Let $\h$ be a real Lie algebra of dimension $\dim \h = 2n$, not necessarily solvable. Recall that a $2$-form $\omega \in \alt^2 \h^*$ on $\h$ is said to be \textit{symplectic} if $\omega^n := \omega \wedge \cdots \wedge \omega$ ($n$ times) is nonzero and
\begin{align*}
    \omega([x,y], z) + \omega([y,z], x) + \omega([z,x], y) = 0 \text{ for all $x$, $y$, $z \in \h$}.
\end{align*}
    \noindent This last condition is equivalent to the fact that $d_{\h} \omega = 0$, meaning that $\omega$ is $d_{\h}$-closed, where $d_{\h}$ is the Chevalley-Eilenberg differential of $\h$ (refer to Section \ref{section: cohomology remarks} for more details). The pair $(\h, \omega)$ is said to be a \textit{symplectic Lie algebra}. If $\omega$ is an $d_{\h}$-exact form, that is, the $d_{\h}$-derivative of some $1$-form $\sigma \in \alt^1 \h^*$, then $\omega$ is called a \textit{Frobenius} form, and the pair $(\h, \omega)$ is called a \textit{Frobenius Lie algebra}. 

     \indent We are interested in unimodular Lie algebras (as in Proposition \ref{prop: milnor}). It turns out that the existence of a Frobenius form is at odds with unimodularity.   
\begin{proposition} \cite[Proposition 3.4]{DM} \label{prop: frobenius implies nonunimodular}
    Frobenius Lie algebras are nonunimodular.
\end{proposition}

    \indent The next result links unimodularity and solvability in the symplectic setting.
\begin{proposition} \cite[Theorem 11]{Chu} \label{prop: chu}
    Unimodular symplectic Lie algebras are solvable. 
\end{proposition} 

\indent In relation to the symplectic $1$-Lefschetz condition, a characterization of unimodular symplectic Lie algebras (thus solvable, as per Proposition \ref{prop: chu}) is already known. One of the main goals in this article is to obtain something as close as possible to this characterization but in the contact setting. We believe both Theorem \ref{thm: main, i guess} and Corollary \ref{cor: boring benson-gordon} below accomplish this in spirit. 
\begin{theorem} \cite[proof of Theorem 2, Remarks in Section 2]{BG2} \label{thm: Benson-Gordon 2}
    A symplectic unimodular Lie algebra $(\h, \omega)$ is $1$-Lefschetz if and only if the following conditions are met:
\begin{enumerate} [\rm (i)]
    \item There is an abelian complement $\a$ in $\h$ of the derived subalgebra $\n := [\h,\h]$. 
    \item Both $\a$ and $\n$ are even-dimensional. 
    \item The center $\z(\h)$ of $\h$ intersects $\n$ trivially. 
    \item The symplectic form $\omega$ is cohomologous to a left-invariant symplectic form $\omega_{\a} + \omega_{\n}$, where $\n = \ker(\omega_{\a})$ and $\a = \ker(\omega_{\n})$. This amounts to the fact that $\a$ and $\n$ are symplectic and $\omega$-orthogonal subspaces of $(\h,\omega)$. 
    \item Both $\omega_{\a}$ and $\omega_{\n}$ are closed but non-exact on $\h$ (also in $\a$ and in $\n$, respectively). 
    \item The adjoint action of $\a$ on $\n$ is by infinitesimal symplectic automorphisms of $(\n, \omega_{\n})$.
\end{enumerate}
\end{theorem}
    \indent Theorem \ref{thm: Benson-Gordon 2} ensures in particular that if $\g$ is the Lie algebra associated to a solvmanifold $\Gamma \backslash G$ and there is an isomorphism $H_{dR}^*(\Gamma \backslash G) \cong H^*(\g)$ (see Proposition \ref{prop: sometimes it is an isomorphism}) then $G$ is a semidirect product $A \ltimes N$, where $A$ is a connected abelian subgroup of $G$ and $N$ is the (nilpotent) commutator subgroup of $G$; moreover, $N$ admits a left-invariant symplectic structure, and the action of $A$ in $N$ is by symplectomorphisms.    

    \indent Some observations leading to the proof of Theorem \ref{thm: Benson-Gordon 2} given in \cite{BG2} are of interest for us, since are also used to establish some of the main results in this article. We defer to Section \ref{section: cohomology remarks} for more details. 

    \indent The situation for nilmanifolds is simpler.  
    
\begin{theorem} \cite[proof of Theorem A]{BG1} \label{thm: Benson-Gordon 1}
    A symplectic nilmanifold is $1$-Lefschetz if and only if it is diffeomorphic to a torus. 
\end{theorem}
    \indent Notice that Theorem \ref{thm: Benson-Gordon 1} follows from Nomizu's theorem (see Proposition \ref{prop: sometimes it is an isomorphism}(i)) and Theorem \ref{thm: Benson-Gordon 2} by noting that, since $\z(\h)$ and $[\h,\h]$ intersect nontrivially for $\h$ nilpotent and nonabelian, and so (iii) fails to hold. However, in \cite{BG1} a different proof is given.  

    \indent One of our main results is an analogous characterization to the one in Theorem \ref{thm: Benson-Gordon 1} in the contact setting, but with a further hypothesis regarding the contact form (see Theorem \ref{thm: odd benson-gordon}).

    \indent Theorem \ref{thm: Benson-Gordon 1} entails that a nilmanifold has a Kähler form if and only if it is diffeomorphic to a torus. In fact, that is how the main result of \cite{BG1} is stated. This particular characterization is known to have many different proofs, including \cite{Hano} (in the invariant setting), \cite{BG1}, \cite{McDuff}, and \cite{Hasegawa 1}; remarkably, the last three proofs were published within a year. A similar characterization is now known for a general solvmanifold: it admits Kähler forms if and only if it is a finite quotient of a complex torus which has the structure of a complex torus bundle over a complex torus; moreover, a solvmanifold of a completely solvable Lie group admits Kähler forms if and only if it is a complex torus (see \cite[Main Theorem]{Hasegawa 2} for the proof of both claims). A similar characterization for Sasakian solvmanifolds is also known, and we describe it in Section \ref{section: solvmanifolds and the contact Lefschetz condition.}. 

    \indent For more information about what is known concerning the Lefschetz condition on symplectic solvmanifolds, see \cite{AG2}. 

\subsection{Solvmanifolds and the contact Lefschetz condition} \label{section: solvmanifolds and the contact Lefschetz condition.} 

    \indent Let $\g$ be a real Lie algebra of dimension $\dim \g = 2n+1$, not necessarily solvable. Following Section \ref{section: the lefschetz condition}, a $1$-form $\eta \in \alt^1 \g^*$ is said to be \textit{contact} if 
\begin{align*}
    \eta \wedge (d \eta)^n \neq 0.
\end{align*}
    \noindent In such case, the pair $(\g, \eta)$ is called a \textit{contact Lie algebra}. As before, there is a unique vector $\xi \in \g$ satisfying
\begin{align*}
    \iota_{\xi} \eta = 1, \quad \iota_{\xi} d \eta = 0. 
\end{align*}
\noindent and it is called the \textit{Reeb vector} of $(\g, \eta)$. Notice that the condition $\iota_{\xi} d \eta = 0$ implies that 
\begin{align*}
    \mathrm{im}(\ad_{\xi}) \subseteq \ker \eta, \quad \ad_{\xi}( \ker \eta) \subseteq \ker \eta. 
\end{align*}
    \indent Recall that the only semisimple Lie algebras admitting a contact form are $\mathfrak{su}(2)$ and $\mathfrak{sl}(2,\R)$ (see \cite[Theorem 5]{BW}), and in particular both of them are $3$-dimensional. Together with Remark \ref{obs: iff unimod, nil, solv} below, this effectively restricts our attention to the solvable case, despite the aim of working in full generality.

    \indent While the material we review in this section is fairly classical, we follow the articles \cite{AFV} and \cite{CF}. Therein, proofs are usually carried over in the Sasakian (respectively, $K$-contact) and Kähler (respectively, almost Kähler) context, but remain true in our more general setting. 
\begin{proposition} \cite[Proposition 1]{AFV} \label{prop: center in contact Lie algebras}
    The center of a contact Lie algebra is either trivial or 1-dimensional. In the latter case, it is generated by the Reeb vector. 
\end{proposition}
\begin{remark} \label{obs: decomposable lie algebras are not contact}
    As a consequence of Proposition \ref{prop: center in contact Lie algebras}, decomposable nilpotent Lie algebras are never contact: as each factor is a nilpotent Lie algebra, the center of the original algebra would have dimension at least $2$. In particular, a decomposable nilmanifold does not have an invariant contact form. This gives an alternative and simpler proof of \cite[Theorem 3.2]{Kutsak}. On the other hand, the deep result in \cite[Theorem 1.1]{BEM} implies that any odd-dimensional parallelizable closed manifold admits contact forms (see \cite[Theorem 4]{Bock}); in particular, decomposable nilmanifolds do have (noninvariant) contact forms. 
\end{remark}
    \indent In this article we are concerned exclusively with contact Lie algebras having nontrivial center, which turn out to be in bijective correspondence with symplectic Lie algebras via a two-way construction process we call \textit{contactization} which we now describe. 
\begin{proposition} \cite[Proposition 2]{AFV} 
    \phantom{.}
\begin{enumerate} [\rm (i)]
    \item If $(\h, \omega)$ is a symplectic Lie algebra then the $1$-dimensional vector space extension $\g := \R \xi \oplus \h$ is made a Lie algebra with bracket
\begin{align} \label{eq: contactization bracket}
    [x,y]_{\g} := \omega(x,y) \xi + [x,y]_{\h} \text{ for all $x$, $y \in \h$}, \quad [\xi, \h] = 0,
\end{align}
    \noindent and the $1$-form $\eta \in \alt^1 \g^*$ given by
\begin{align*}
    \eta(a \xi + x) = a, \text{ for $a \in \R$ and $x \in \g$}
\end{align*}
    \noindent is a contact form on $\g$. In particular, $\xi \in \g$ is the Reeb vector on $(\g, \eta)$ and $\z(\g) = \R \xi$. Also, $\omega = - d_{\g} \eta$. 
    \item If $(\g, \eta)$ is a contact Lie algebra with nontrivial center then the pair $(\h, \omega)$ given by $\h := \ker \eta$ and $\omega := (- d_{\g} \eta) \vert_{\h}$ is a symplectic Lie algebra of dimension $2n$ with bracket
\begin{align*}
    [\cdot, \cdot]_{\h} := p_{\h} \circ [\cdot, \cdot]_{\g},
\end{align*}
     \noindent where $p_{\h}:\R \xi \oplus \h \to \h$ is the canonical projection. Moreover, equation \eqref{eq: contactization bracket} holds.
\end{enumerate}
\end{proposition}
    \indent Succinctly, the process of contactization is just a $1$-dimensional central extension by a symplectic $2$-cocycle. This correspondence is a well established result of homological algebra. We employ the notation $\g := \R \xi \oplus_{\omega} \h$.
\begin{corollary} \cite[Proposition 4]{AFV} \label{prop: isomorphism when center}
    Two contact Lie algebras $(\g_1, \eta_1)$ and $(\g_2, \eta_2)$ with nontrivial centers are isomorphic as contact Lie algebras if and only if $(\ker \eta_1, - d \eta_1 \vert_{\ker \eta_1})$ and $(\ker \eta_2, - d \eta_2 \vert_{\ker \eta_2})$ are isomorphic as symplectic Lie algebras.
\end{corollary}
\begin{remark} \label{obs: iff unimod, nil, solv}
    Let $(\g,\eta)$ be a contact Lie algebra arising via contactization from a symplectic Lie algebra $(\h,\omega)$. Cartan's criterion shows that $\g$ is solvable if and only if $\h$ is solvable. Since $\xi$ is central in $\g$, it follows that $\ad_{\xi}^{\g} = 0$ and 
\begin{align*}
    \ad_x^{\g} = 
    \left[
	\begin{array}{c | c}      
		0 & * \\
        \hline
        0 & \ad_x^{\h}
	\end{array} \right] 
\end{align*}
    \noindent for all $x \in \h$ viewed also as an element of $\g$. In particular, $\g$ is unimodular if and only if $\h$ is unimodular and $\g$ is completely solvable if and only if $\h$ is completely solvable; also, according to Engel's theorem, $\g$ is nilpotent if and only if $\h$ is nilpotent. Imposing $\g$ to be unimodular then forces $\h$ to be unimodular; since $\h$ is also symplectic and thus solvable following Proposition \ref{prop: chu}, we obtain that $\g$ must be solvable as well.
\end{remark}
 
    \indent An analogous construction to that mentioned in Section \ref{section: the lefschetz condition} ensures that any contact Lie algebra $(\g, \eta)$, whatever its center, admits a compatible contact metric structure $(\eta, g, \Phi)$. The formulas furnishing the compatibility are, of course, similar to those appearing before. There are analogues to Propositions \ref{prop: K-contact equivalences} and \ref{prop: Sasaki equivalences}, and thus the definitions of $K$-contact Lie algebras and Sasakian Lie algebras are clear. Recalling that  
\begin{align*}
    \mathcal{L}_{\xi} \Phi = [\ad_{\xi}, \Phi], \quad (\mathcal{L}_{\xi} g)(x,y) = g(\ad_{\xi} x,y) + g(x, \ad_{\xi} y) \text{ for all $x$, $y \in \g$}, 
\end{align*}
    \noindent where $\mathcal{L}_{\xi}$ denotes the Lie derivative with respect to $\xi$, we see that there are more equivalent statements for the $K$-contact condition in the Lie-theoretic version of Proposition \ref{prop: K-contact equivalences}. In order to state them, denote by $\nabla$ the Levi-Civita connection associated to the inner product $g$ on $\g$. 
\begin{proposition} \label{prop: $K$-contact equivalences, 2nd version}
    The following conditions are equivalent:
\begin{multicols}{2}
\begin{enumerate} [\rm (i)]
    \item $\mathcal{L}_{\xi} g = 0$. 
    \item $\mathcal{L}_{\xi} \Phi = 0$. 
    \item $\ad_{\xi}$ is skew symmetric with respect to $g$.
    \item $\ad_{\xi}$ and $\Phi$ commute. 
    \item $\ad_{\xi} \circ \Phi$ is symmetric with respect to $g$.
    \item $\Phi x = - \nabla_x \xi$ for all $x \in \g$. 
\end{enumerate}    
\end{multicols}
\end{proposition}
    \indent In particular, both $\ker \ad_{\xi}$ and $\im \ad_{\xi}$ are $\Phi$-invariant subspaces. There appears to be no new equivalent statements for the Sasakian condition in the Lie-theoretic version of Proposition \ref{prop: Sasaki equivalences}.  
\begin{remark} \label{obs: rather trivial}
    Contact Lie algebras $(\g,\eta)$ with nontrivial center satisfy $\ad_{\xi} = 0$, as Proposition \ref{prop: center in contact Lie algebras} shows, and therefore are trivially $K$-contact: any compatible metric does the job.    
\end{remark} 
\indent The rather trivial Remark \ref{obs: rather trivial} plays a role in the statement of Corollary \ref{cor: cagliari} below, which generalizes the results found in \cite{Cagliari 1}. 

    \indent It is easy to characterize both Sasakian and $K$-contact Lie algebras with nontrivial center. They are in correspondence with Kähler and almost Kähler Lie algebras, respectively. 
\begin{proposition} \cite[Corollary 3]{AFV} \cite[Theorem 3.6]{CF} \label{prop: center = (Sasaki = Kahler)}
    Let $(\g,\eta)$ be a contact Lie algebra arising as the contactization of a symplectic Lie algebra $(\h, \omega)$. Then $(\g, \eta)$ is $K$-contact if and only if $(\h, \omega)$ is almost Kähler, and $(\g, \eta)$ is Sasakian if and only if $(\h, \omega)$ is Kähler. 
\end{proposition}   
    \indent We now describe the standard Sasakian Lie algebra.
\begin{example} \label{ex: heisenberg is sasakian}
    Let $\h_{2n+1}$ be the Lie algebra spanned by $\{X_1, \ldots, X_n, Y_1, \ldots, Y_n, Z\}$, with Lie bracket given by
\begin{align*}
    [X_i, Y_i] = Z \text{ for all $1 \leq i \leq n$}.
\end{align*}
    \noindent $\h_{2n+1}$ is called the real $(2n+1)$-dimensional \textit{Heisenberg Lie algebra}. Let $g$ the inner product obtained by declaring the basis above orthonormal. Set $\xi := Z$, and let $\eta$ be the $1$-form dual to $\xi$ via $g$. Define $\Phi\colon \h_{2n+1} \to \h_{2n+1}$ by
\begin{align*}
    \Phi(Z) = 0, \quad \Phi(X_i) = Y_i, \quad \Phi(Y_i) = - X_i \quad \text{ for all $1 \leq i \leq n$}. 
\end{align*}
    \noindent It is straightforward to check that $(\eta, g, \Phi)$ is a compatible Sasakian structure on $\h_{2n+1}$. 
\end{example}
    \indent It turns out, there are no more Sasakian nilpotent Lie algebras than those described in Example \ref{ex: heisenberg is sasakian}.
\begin{proposition} \cite[Theorem 3.9]{AFV} \label{prop: for later use}
    The only contact nilpotent Lie algebras admitting a compatible Sasakian structure are the Heisenberg Lie algebras.
\end{proposition}
    \indent In combination with Remark \ref{obs: rather trivial}, Proposition \ref{prop: for later use} guarantees that there are many $K$-contact Lie algebras that have no compatible Sasakian structure: any non-Heisenberg nilpotent contact Lie algebra gives an example. 

    \indent Certainly, the Sasakian structure on $\h_{2n+1}$ described in Example \ref{ex: heisenberg is sasakian} gives rise to a left-invariant Sasakian structure on the Heisenberg Lie group $H_{2n+1}$, the corresponding simply connected Lie group associated to $\h_{2n+1}$, and thus on every Heisenberg nilmanifold $\Gamma \backslash H_{2n+1}$. Of course, Proposition \ref{prop: for later use} implies that every nilmanifold endowed with an invariant contact form admits an invariant Sasakian structure if and only if it is a Heisenberg nilmanifold. Surprisingly, the same is true even without the further restriction to invariant contact structures. 
\begin{theorem} \cite[Theorem 1.1]{Cagliari 3} \label{thm: sasakian nilmanifolds}
    A nilmanifold admits a Sasakian structure (not necessarily invariant) if and only if it is a Heisenberg nilmanifold.
\end{theorem}    
    \indent See also \cite[Theorem 8.2]{Kasuya3} for an alternative proof of Theorem \ref{thm: sasakian nilmanifolds}. 

    \indent Theorem \ref{thm: sasakian nilmanifolds} is the clear analogue in the contact setting of the characterization of Kähler nilmanifolds arising from Theorem \ref{thm: Benson-Gordon 1}. Just as in the symplectic case, there is also a characterization of general solvmanifolds admitting Sasakian structures (not necessarily invariant): they are precisely finite quotients of Heisenberg nilmanifolds; moreover, it is also known that a completely solvable solvmanifold admits Sasakian structures if and only if it is diffeomorphic to a Heisenberg nilmanifold (see \cite[Corollary 1.4, 1.5]{Kasuya2} for a proof of both claims). We mention in passing that a similar characterization holds in the context of compact aspherical Sasakian manifolds, taking into account the solvability class of the fundamental group of such a manifold (see \cite[Theorem 1.1, Corollary 1.2]{dNY}). 

    \indent It is then natural to explore whether Theorem \ref{thm: Benson-Gordon 1}, stated in regards to the Lefschetz condition, also holds for contact nilmanifolds. It is also reasonable to look for a characterization of contact $1$-Lefschetz solvmanifolds in some way analogous to Theorem \ref{thm: Benson-Gordon 2}. Seeking to answer the first question, the authors in \cite{Cagliari 1} describe two nilmanifolds endowed with invariant contact forms that are not $1$-Lefschetz (thus, neither Sasakian nor Heisenberg) and admit compatible $K$-contact structures. We show in Theorem \ref{thm: odd benson-gordon} and Corollary \ref{cor: cagliari} below that there is nothing unusual with these examples, and that the same is true for general nilmanifolds with invariant contact forms. To achieve this, we find a characterization of $1$-Lefschetz contact Lie algebras with nontrivial center together with a corresponding result for solvmanifolds (see Theorem \ref{thm: main, i guess} and Remark \ref{obs: main, i guess} below), which is the closest we can get to answering the second question. 
 
\section{The 1-Lefschetz condition on contact solvmanifolds}

\subsection{Cohomology remarks} \label{section: cohomology remarks}

\indent Let $\g$ be a real $n$-dimensional Lie algebra, and let $\n := [\g, \g]$ be its commutator subalgebra. Take $\a$ to be any vector subspace complement of $\n$ in $\g$, so that $\g = \a \oplus \n$. Write $\g^* = \a^* \oplus \n^*$, where
\begin{align*}
    \a^* := \{ \eta \in \g^* \mid \eta(\n) = 0 \}, \quad \n^* := \{ \eta \in \g^* \mid \eta(\a) = 0 \}. 
\end{align*}
\noindent Note that $\a^*$ is independent of the choice of $\a$. Set $l := \dim(\n)$ and $k := \dim(\a)$, and so $n = k+l$. Denote
\begin{align*}
    \alt^{i,j} \g^*:= \alt^i \a^* \otimes \alt^j \n^* \text{ for all $0 \leq i,j \leq n$};
\end{align*}
\noindent in particular, for all $0 \leq p \leq n$,
\begin{align*}
    \alt^p \g^* = \bigoplus_{i + j = p} \alt^{i,j} \g^*.
\end{align*}
\indent Recall that the standard cochain complex $(\alt^* \g^*, d_{\g})$ has exterior derivative $d_{\g}$ completely determined by the conditions
\begin{gather*}
    d_{\g}\theta(x,y) = - \theta([x,y]), \quad \text{for all $\theta \in \g^*$ and all $x$, $y \in \g$}, \\
    d_{\g} (\alpha \wedge \beta) = d_{\g} \alpha \wedge \beta + (-1)^{k} \alpha \wedge d_{\g} \beta, \quad \text{for all } \alpha\in\alt^k\g^*, \, \beta \in \alt^* \g^*.
\end{gather*}
\noindent We write $d$ instead of $d_{\g}$ when there is no risk of confusion. Since the differentials on cochain complexes are essentially dual to the Lie brackets, it readily follows that Lie algebra morphisms are precisely those inducing morphisms of cochain complexes, and viceversa. 
\begin{lemma} \label{lemma: equivalencia para morfismos}
    A linear map $f\colon\g \to \h$ between two Lie algebras $\g$ and $\h$ is a Lie algebra morphism if and only if the induced map $f^*\colon \alt^* \h^* \to \alt^* \g^*$ satisfies $d_{\g} \circ f^* = f^* \circ d_{\h}$.
\end{lemma}
\indent This is the setting in which Theorem \ref{thm: Benson-Gordon 2} is proven in \cite{BG2}. We now rederive some preliminary results from there, as they are useful for our purposes. For all $0 \leq k \leq n$, denote by $b_k(\g)$ the $k$-th Betti number of $\g$; that is, $b_k(\g) := \dim H^k(\g)$. The following result is clear.  
\begin{lemma} \label{lemma: H1 is easy} 
   \phantom{.} 
\begin{multicols}{2}
\begin{enumerate} [\rm (i)]
    \item $d( \alt^{1,0} \g^*) = 0$.
    \item $d:\alt^{0,1} \g^* \to \alt^2 \g^*$ is injective. 
    \item $H^1(\g)$ is identified with $\alt^{1,0} \g^*$. 
    \item $b_1(\g) = \dim \a$. 
\end{enumerate}    
\end{multicols}
\end{lemma}
\indent By Lemma \ref{lemma: H1 is easy}(i) and (ii), the space of closed $1$-forms $Z^1(\g)$ is precisely $\alt^{1,0} \g^*$. Since there are no nonzero exact $1$-forms on $\g$, the identification between $H^1(\g)$ and $\alt^{1,0} \g^*$ in Lemma \ref{lemma: H1 is easy}(iii) is realized by the restriction of the canonical projection map $\alt^1 \g^* \to H^1(\g)$ to $Z^1(\g)$. 

\indent Recall that $\g$ is said to be unimodular if $\tr(\ad_x) = 0$ for any $x \in \g$. 
\begin{lemma} \label{lemma: H2n-1 is almost easy}
    Let $\Omega \in \alt^n \g^*$ be a $n$-form on $\g$.
\begin{enumerate} [\rm (i)]
    \item For all $\lambda \in \alt^{n-1} \g^*$ there exists a unique $x \in \g$ such that $\lambda = \iota_x \Omega$.
    \item If $\lambda \in \alt^{n-1} \g^*$ is written as $\lambda = \iota_x \Omega$ then $d \lambda = - \tr( \ad_x ) \Omega$. 
    \item $\g$ is unimodular if and only if $d\colon \alt^{n-1} \g^* \to \alt^n \g^*$ is the zero map. 
    \item $\g$ is unimodular if and only if $H^n(\g)$ is nonzero. 
\end{enumerate}  
\end{lemma}
\begin{proof}
    Notice that (i) is immediate; moreover, (iii) follows directly from (ii), and (iv) follows directly from (iii) since $H^n(\g)$ is either trivial or $1$-dimensional. The verification of (ii) results from a straightforward computation in terms of a basis of $\g$, so we omit it. \qedhere 

\end{proof}
\indent The cohomology of unimodular Lie algebras exhibits Poincaré duality. Indeed, after choosing a $n$-form $\Omega$ on $\g$, Lemma \ref{lemma: H2n-1 is almost easy}(iv) allows for an identification $H^{n}(\g) \cong \R [\Omega]$. Poincaré duality then amounts to the fact that, for all $0 \leq k \leq n$, the wedge product induces a \textit{non-degenerate} bilinear pairing
\begin{gather*}
    H^k(\g) \times H^{n-k}(\g) \to \R, 
\end{gather*}
\noindent resulting from the composition
\begin{align*}
    [(\alpha_1], [\alpha_2]) \mapsto [\alpha_1 \wedge \alpha_2] = c \, [\Omega] \mapsto c.  
\end{align*}
\noindent In fact, the nondegeneracy of one of these pairings is enough to ensure the nondegeneracy of all of them. We refer to \cite[Chapter 1, Section 3, page 27]{Fuks} for details; see also the discussion at the end of this section. As an immediate consequence, the Betti numbers satisfy $b_k(\g) = b_{n-k}(\g)$ for all $0 \leq k \leq n$.  

\indent The next result is established in \cite{BG2} through a slightly different, computation-oriented proof.
\begin{lemma} \label{lemma: H2n-1 is easy for unimodular}
    Let $\g$ be unimodular, and pick any volume cochain $\Omega \in \alt^n \g^*$. 
\begin{multicols}{2}
\begin{enumerate} [\rm (i)]
    \item $d(\alt^{k-1,l} \g^*) = 0$ and $d( \alt^{k,l-1} \g^*) = 0$. 
    \item Every element of $\alt^{k, l-1} \g^*$ is exact.  
    \item $H^{n-1}(\g)$ is identified with $\alt^{k-1,l} \g^*$.
    \item $H^{n-1}(\g)$ is identified with $\{ \iota_x \Omega \mid X \in \a \}$. 
    \item $b_{n-1}(\g) = \dim \a$. 
\end{enumerate}     
\end{multicols}
\end{lemma}
\begin{proof}
    Notice that (i) follows from Lemma \ref{lemma: H2n-1 is almost easy}(iii), and (iv) follows from (iii) and Lemma \ref{lemma: H2n-1 is almost easy}(i). Although (v) also follows from (iii), it is convenient to view it as just a direct consequence of Lemma \ref{lemma: H1 is easy}(iv) together with Poincaré duality to avoid circularity in our argument. Now, the bilinear pairing arising from the composition
\begin{gather*}
    \alt^{1,0} \g^* \times \alt^{k-1,l} \g^* \to \alt^{k,l} \g^* = \alt^n \g^* \cong \R, \quad (\alpha_1, \alpha_2) \mapsto \alpha_1 \wedge \alpha_2 = c \, \Omega \mapsto c  
\end{gather*}
    is non-degenerate, essentially because for any $x \in \g$ and any $\sigma \in \g^*$ that is dual to $x$ (in the sense that $\sigma(x) = 1$) it holds that $\sigma \wedge \iota_x \Omega = \Omega$, since $\Omega$ is a top form. Recall that $H^1(\g)$ and $\alt^{1,0} \g^*$ are identified as per Lemma \ref{lemma: H1 is easy}(iii). We have already observed that any element of the set $\alt^{k-1,l} \g^*$ is closed, and so they define cohomology classes in $H^{n-1}$. The non-degeneracy of the pairing ensures that no element of $\alt^{k-1,l} \g^*$ is exact, for otherwise the induced pairing on cohomology would be degenerate, contradicting Poincaré duality. Since  $\dim \alt^{k-1,l} \g^* = \dim \a = b_{n-1}(\g)$ because of (v), (ii) and (iii) are thus established.
\end{proof}
\indent As with Lemma \ref{lemma: H1 is easy}(iii), the identification between $H^{n-1}(\g)$ and $\alt^{k-1,l} \g^*$ is done through the restriction of the canonical projection map $\alt^{n-1} \g^* \to H^{n-1}(\g)$ to $\alt^{k-1,l} \g^*$, which contains only closed non-exact forms according to Lemma \ref{lemma: H2n-1 is easy for unimodular}(i) and (ii). 


\subsection{The 1-Lefschetz condition for contact Lie algebras arising from contactization} \label{section: the 1-Lefschetz condition for contact Lie algebras arising from contactization}

\indent Throughout this section, let $(\g,\eta)$ be a contact Lie algebra with nontrivial center arising as the contactization of a symplectic Lie algebra $(\h, \omega)$. In symbols, $\g = \R \xi \oplus_{\omega} \h$. Set $\dim \g = 2n+1$ and $\dim \h = 2n$. The cochain complexes of $\g$ and $\h$ are denoted as
\begin{align*}
    (\alt^* \g^*, d_{\g}), \quad (\alt^* \h^*, d_{\h}),
\end{align*}
\noindent respectively. Recall from Section \ref{section: solvmanifolds and the contact Lefschetz condition.}  that, in this scenario, $\z(\g) = \R \xi$, the condition $\iota_{\xi} \eta = 1$ uniquely determines $\xi$, and $\omega = - d_{\g} \eta$. 

\indent Let $\pi\colon\g \to \h$ be the canonical projection, which is a Lie algebra morphism. According to Lemma \ref{lemma: equivalencia para morfismos}, the pullback $\pi^*\colon\alt^* \h^* \to \alt^* \g^*$ satisfies $d_{\g} \circ \pi^* = \pi^* \circ d_{\h}$, and so it induces a map in cohomology, denoted by the same name. Notice that $\pi^*$ preserves degrees, both at the level of forms and cohomology. Since $\pi$ is surjective, there exists a linear map $s\colon\h \to \g$ satisfying $\pi \circ s = \Id_{\h}$, and therefore inducing a right inverse $s^*\colon \alt^* \g^* \to \alt^* \h^*$ of $\pi^*$ at the level of forms. However, since $s$ is not generally a Lie algebra morphism, $\pi^*$ is not necessarily invertible on cohomology. As $\h$ is a subset of $\g$, a section $s$ is given by the inclusion $\h \hookrightarrow \g$. 

\indent Fix $0 \leq k \leq 2n$ and pick any $\alpha \in \alt^k \g^*$. There are unique forms $\alpha_L \in \alt^{k-1} \g^*$ and $\alpha_R \in \alt^k \g^*$, determined uniquely by the conditions $\iota_{\xi} \alpha = \alpha_L$ and $\iota_{\xi} \alpha_R = 0$, such that 
\begin{align} \label{eq: left and right parts}
    \alpha = \eta \wedge \alpha_L + \alpha_R. 
\end{align}
\noindent Notice in particular that $\iota_{\xi} \alpha_L = 0$ as well. The fact that $\iota_{\xi} \alpha_L = 0$ and $\iota_{\xi} \alpha_R = 0$ means that both $\alpha_L$ and $\alpha_R$ are actually pulled back from forms on $\h$ via $\pi^*$. This allows for an identification of $(\alt^* \h^*, d_{\h})$ as the subcomplex of $(\alt^* \g^*, d_{\g})$ given by
\begin{align} \label{eq: alt h is subcomplex of alt g}
    \alt^k \h^* \cong \{ \alpha \in \alt^k \g^* \mid \iota_{\xi} \alpha = 0\}, \quad 0 \leq k \leq 2n, 
\end{align}
\noindent and moreover $d_{\h}$ coincides with the restriction of $d_{\g}$ to $\alt^* \h^*$. This last bit is true essentially because $d_{\g} \circ \pi^* = \pi^* \circ d_{\h}$, but we also give a proof based on the identification in \eqref{eq: alt h is subcomplex of alt g}.
\begin{lemma} \label{lemma: they are the same differential}
    If $\beta \in \alt^* \g^*$ satisfies $i_{\xi} \beta = 0$ then $d_{\g} \beta = d_{\h} \beta$. 
\end{lemma}
\begin{proof}
    If $\beta\in \alt^1\g^*$ then, for any $x$, $y \in \h \subseteq \g$,
\begin{align*}
    - (d_{\h} \beta)(x,y) = \beta( [x,y]_{\h} ) = \beta( \omega(x,y) \xi + [x,y]_{\h} ) = \beta( [x,y]_{\g} ) = - (d_{\g} \beta)(x,y).
\end{align*}
    \noindent Since both $d_{\g}$ and $d_{\h}$ are determined by their actions on $1$-forms, the result follows from here.
\end{proof}    
\indent From now on, assume further that $\g$ is unimodular, which ensures that $\h$ is unimodular as well due to Remark \ref{obs: iff unimod, nil, solv}. Less trivial implications are that the symplectic form $\omega$ on $\h$ is not exact, as observed in Proposition \ref{prop: frobenius implies nonunimodular}, and that both $\h$ and $\g$ are solvable, as per Remark \ref{obs: iff unimod, nil, solv} once more. Notice that, in either the symplectic or the contact sense, unimodularity is in fact equivalent to being $0$-Lefschetz (as we can see from Lemma \ref{lemma: H2n-1 is almost easy}), and a necessary condition for the $1$-Lefschetz property to hold.
\begin{lemma} \label{lemma: H1(h) and H1(g) are isomorphic}
    $\pi^*\colon H^1(\h) \to H^1(\g)$ is an isomorphism. 
\end{lemma}
\begin{proof}
    Since $\pi^*\colon \alt^1 \h^* \to \alt^1 \g^*$ and $d_{\g} \circ \pi^* = \pi^* \circ d_{\h}$, the restriction $\pi^*\colon Z^1(\h) \to Z^1(\g)$ is well-defined. Moreover, it is injective since it is the restriction of an injective map. Pick any $\gamma \in Z^1(\g)$, and write it as $\gamma = a \eta + \gamma'$ for some $a\in \R$ and $\gamma' \in \alt^1 \h^*$, following the decomposition in \eqref{eq: left and right parts}. Using Lemma \ref{lemma: they are the same differential} and that $d_{\g} \eta = - \omega$, one gets that 
\begin{align*}
    0 = d_{\g} \gamma = a d_{\g} \eta + d_{\g} \gamma' = - a \omega + d_{\h} \gamma', 
\end{align*}
    \noindent or equivalently $a \omega = d_{\h} \gamma'$. As $\omega$ is not exact as per Proposition \ref{prop: frobenius implies nonunimodular}, it follows that $a = 0$ and $\gamma' \in Z^1(\h)$. Therefore, $\gamma = \gamma'=\pi^* \gamma'$, and $\pi^*\colon Z^1(\h) \to Z^1(\g)$ is surjective. Since there are no exact $1$-forms, $Z^1(\h)$ and $Z^1(\g)$ are identified with $H^1(\h)$ and $H^1(\g)$ respectively, and the claim follows. 
\end{proof}
\begin{corollary} \label{cor: un cuidado para definir el mapa lefschetz contacto}
    $(d \eta)^n \wedge \alpha = 0$ and $\eta \wedge (d \eta)^{n-1} \wedge \alpha$ is closed for all $\alpha \in Z^1(\g)$. 
\end{corollary}
\begin{proof}
    Since $Z^1(\g) = H^1(\g)$, it follows from Lemma \ref{lemma: H1(h) and H1(g) are isomorphic} that any $\alpha \in Z^1(\g)$ is actually the pullback of a closed $1$-form on $\h$; similarly, $(d \eta)^n$ is the pullback of $\omega^n \in \alt^{2n} \h^*$. Therefore, for any $\alpha \in Z^1(\g)$, $(d \eta)^n \wedge \alpha$ is the pullback of a closed form on $\h$, and moreover, it has degree $2n+1$ while $\dim \h = 2n$. Thus, $(d \eta)^n \wedge \alpha$ is the pullback of the zero form on $\h$, and so zero itself. Therefore,
\begin{align*}
    d (\eta \wedge (d \eta)^{n-1} \wedge \alpha) = (d \eta) \wedge (d \eta)^{n-1} \wedge \alpha = (d \eta)^n \wedge \alpha = 0
\end{align*}
   \noindent for all $\alpha \in Z^1(\g)$. 
\end{proof}

\indent Set $\n_{\g} := [\g,\g]_{\g}$ and $\n_{\h} := [\h,\h]_{\h}$. As in Section \ref{section: cohomology remarks}, choose vector space complements
\begin{align*}
    \g = \a_{\g} \oplus \n_{\g}, \quad \h = \a_{\h} \oplus \n_{\h},
\end{align*} 
\begin{proposition} \label{prop: xi belongs in the commutator}
    $\n_{\g} = \R \xi \oplus \n_{\h}$ as vector spaces. In particular, $\xi \in \n_{\g}$.
\end{proposition}
\begin{proof}
    Combining Lemma \ref{lemma: H1 is easy}(iv) and Lemma \ref{lemma: H1(h) and H1(g) are isomorphic}, it follows that
\begin{align*}
    \dim \a_\h = b_1(\h) = b_1(\g) = \dim \a_{\g},
\end{align*}
\noindent and hence $\dim \n_{\h} + 1 = \dim \n_{\g}$. Moreover, since
\begin{align*}
    [x,y]_{\g} = \omega(x,y) \xi + [x,y]_{\h} \in \R\xi \oplus \n_{\h}
\end{align*}
\noindent for all $x$, $y \in \h$, one gets $\n_{\g} \subseteq \R \xi \oplus \n_{\h}$. The equality $\n_{\g} = \R \xi \oplus \n_{\h}$ follows from dimension counting. 
\end{proof}
\begin{remark} \label{obs: heisenberg are precisely central extensions of abelian lie algebras}
    Proposition \ref{prop: xi belongs in the commutator} readily implies that $\dim \n_{\g} = 1$ if and only if $\h$ is abelian, from which it follows that Heisenberg Lie algebras are precisely those arising as $1$-dimensional central extensions of abelian Lie algebras. 
\end{remark}
\begin{remark}
    The fact that $\xi \in \n_\g$ can be established without appeal to cohomological considerations if we further assume that $\g$ is nilpotent, for in this case $\z(\g)$ and $\n_{\g}$ intersect nontrivially and, as pointed out in Proposition \ref{prop: center in contact Lie algebras}, $\z(\g)$ is generated by $\xi$. 
\end{remark} 
\begin{remark} \label{obs: ok, we do use this fact in the sequel}
    The fact that $\xi \in \n_{\g}$ in Proposition \ref{prop: xi belongs in the commutator} can also be obtained by appealing to the universal coefficient theorem for $\h$, and in fact is equivalent to the condition that $\omega$ is not a Frobenius form on $\h$. Recall that since we are working over $\R$, the universal coefficient theorem amounts to the fact that $H^k(\h) \cong \mathrm{Hom}(H_k(\h), \R)$ as vector spaces for all $0 \leq k \leq \dim \h$. Thus $[\omega]$ is a nonzero cohomology class of degree $2$ (i.e., $\omega$ is not a Frobenius form) if and only if it is dual to some nonzero homology class $[Z] \in H_2(\h)$ of degree $2$, which is equivalent to the assertion that there are some $x_i$'s and $y_i$'s in $\h$ such that $\sum_i [x_i, y_i]_{\h} = 0$ and $\sum_i \omega(x_i, y_i) \neq 0$, and thus equivalent to the fact that $\xi \in \n_{\g}$.
\end{remark}

\indent As per Proposition \ref{prop: xi belongs in the commutator}, it is always possible to take $\a_{\h}$ and $\a_{\g}$ to be equal. Thus, from now on we set $\a := \a_{\h} = \a_{\g}$ and
\begin{align*}
    \g = \a \oplus \n_{\g}, \quad \h = \a \oplus \n_{\h},
\end{align*}
\noindent with the extra knowledge that $\n_{\g} = \R \xi \oplus \n_{\h}$. From now on we fix the volume forms $\Omega_{\g}$ and $\Omega_{\h}$ on $\g$ and $\h$ respectively to be
\begin{align*}
    \Omega_{\g} = \eta \wedge (d \eta)^n, \quad \Omega_{\h} = (d \eta|_\h)^n.
\end{align*}
\noindent Notice in particular that $\Omega_{\g} = \eta \wedge \Omega_{\h}$ on $\g$. Set $k:=\dim \a$, $l_{\g} := \dim \n_{\g}$, and $l_{\h} := \dim \n_{\h}$. As in Section \ref{section: cohomology remarks}, write
\begin{align*}
    \alt^{k-1, l_{\g}} \g^* = \{ \iota_x \Omega_{\g} \mid x \in \a \}, \quad \alt^{k-1, l_{\h}} \h^* = \{ \iota_x \Omega_{\h} \mid x \in \a \}. 
\end{align*} 
\begin{lemma} \label{lemma: we are almost done}
    The map $H^{2n-1}(\h) \to H^{2n}(\g)$ given by $[\beta] \mapsto [\eta \wedge \beta]$ is an isomorphism, with inverse $H^{2n}(\g) \to H^{2n-1}(\h)$ given by $[\alpha] \to [\iota_{\xi} \alpha]$. 
\end{lemma}
\begin{proof}
    Proposition \ref{prop: xi belongs in the commutator} shows that $\xi \in \n_{\g}$, and in particular that $\xi \notin \a$. Therefore,  
\begin{align*}
    \iota_x( \eta \wedge \lambda) = \iota_x \eta \wedge \lambda - \eta \wedge \iota_x \lambda = - \eta \wedge \iota_x \lambda
\end{align*}
    \noindent for any $\lambda \in \alt^* \g^*$ and any $x \in \a$, and consequently
\begin{align*}
    \alt^{k-1, l_{\g}} \g^* = \{ \iota_x \Omega_{\g} \mid x \in \a \} = \{ - \eta \wedge \iota_x \Omega_{\h} \mid x \in \a \} = (-\eta) \wedge \alt^{k-1, l_{\h}} \h^*. 
\end{align*}
    \noindent Therefore, the maps
\begin{align*}
\begin{array}{ccc}
    \alt^{k-1, l_{\h}} \h^* \to \alt^{k-1, l_{\g}} \g^*, & & \alt^{k-1, l_{\g}} \g^* \to \alt^{k-1, l_{\h}} \h^*, \\
     \beta \mapsto \eta \wedge \beta, & & \alpha \mapsto \iota_{\xi} \alpha,
\end{array}    
\end{align*}
    \noindent are well defined. It is also clear that the first map is a surjection, and so an isomorphism since both spaces are $k$-dimensional; the inverse is given by the second map since 
\begin{align*}
    \beta \mapsto \eta \wedge \beta \mapsto \iota_{\xi} (\eta \wedge \beta) = \iota_{\xi}\eta \wedge \beta - \eta \wedge \iota_{\xi} \beta = \beta
\end{align*}
    \noindent for all $\beta \in \alt^{k-1, l_{\h}} \h^*$, as $\iota_{\xi} \eta = 1$ and $\iota_{\xi} \beta = 0$. Recall from Lemma \ref{lemma: H2n-1 is easy for unimodular}(iv) that $H^{2n}(\g)$ and $H^{2n-1}(\h)$ are identified with $\alt^{k-1, l_{\g}} \g^*$ and $\alt^{k-1, l_{\h}} \h^*$ through the maps $\pi_{\g}\colon \alt^{k-1, l_{\g}} \g^* \to H^{2n}(\g)$ and $\pi_{\h}\colon \alt^{k-1, l_{\h}} \h^* \to H^{2n-1}(\h)$ arising as restrictions of the respective canonical projections. This means that the maps
\begin{align*}
\begin{array}{ccc}
    H^{2n-1}(\h) \to H^{2n}(\g), & & H^{2n}(\g) \to H^{2n-1}(\h), \\
   \phantom{.} [\beta] \mapsto \pi_{\g} ( \eta \wedge \pi_{\h}^{-1} ([\beta])), & &  [\alpha] \mapsto \pi_{\h} \circ \iota_{\xi} \circ \pi_{\g}^{-1} ([\alpha]),
\end{array}    
\end{align*}
    \noindent are well defined and inverse to each other. The claim is thus established. 
\end{proof}
    \indent Lemma \ref{lemma: we are almost done} can be stated succinctly as follows: 
\begin{align*}
    H^{2n}(\g) \cong [\eta] \wedge H^{2n-1}(\h), \quad \text{via $[\eta] \wedge (-)$}.
\end{align*}

    \indent Lemma \ref{lemma: H1(h) and H1(g) are isomorphic} and the first part of Corollary \ref{cor: un cuidado para definir el mapa lefschetz contacto} ensure that any cohomology class in $H^1(\g)$ admits $\xi$-horizontal and primitive representatives, respectively. According to the remarks following Definition \ref{def: contact lefschetz condition}, this is half of the conditions required for the contact $1$-Lefschetz condition to hold. The second part of Corollary \ref{cor: un cuidado para definir el mapa lefschetz contacto} ensures that the contact $1$-Lefschetz map, defined as
\begin{align*}
    \mathrm{Lef}:H^1(\g) \to H^{2n}(\g), \quad \mathrm{Lef}( [\alpha] ) := [\eta \wedge (d \eta)^{n-1} \wedge \alpha],
\end{align*}
    \noindent is well defined. We refrain calling it $\mathrm{Lef}^{n-1}$ in an effort not to overcomplicate the notation, as we are only concerned with the degree-one case. Following Definition \ref{def: contact lefschetz condition}, $(\g, \eta)$ is $1$-Lefschetz if and only if $\mathrm{Lef}$ is bijective. 

    \indent Consider also the symplectic $1$-Lefschetz map, defined as 
\begin{align*}
    L:H^1(\h) \to H^{2n-1}(\h), \quad L([\beta]) = [\omega^{n-1} \wedge \beta].
\end{align*}
\noindent Once again, we refrain calling it $L^{n-1}$ as there is no risk of confusion. After identifying $d \eta$ and $- \omega$ on $\h$, we obtain the following commutative diagram:  
\begin{center}
    \begin{tikzpicture}[auto]
    \node (A) at (0,0) {$H^1(\h)$};
    \node (B) at (3.5,0) {$H^{2n-1}(\h)$};
    \node (C) at (0,-2) {$H^1(\g)$};
    \node (D) at (3.5,-2) {$H^{2n}(\g)$};
    \draw[->] (A) to node {$L$} (B);
    \draw[->] (A) to node {$\pi^*$} (C);
    \draw[->] (B) to node {$[\eta] \wedge -$} (D);
    \draw[->] (C) to node {$(-1)^n \mathrm{Lef}$} (D); 
\end{tikzpicture}
\end{center}    
\noindent The vertical arrows correspond with the maps in Lemma \ref{lemma: H1(h) and H1(g) are isomorphic} and in Lemma \ref{lemma: we are almost done}, and both are isomorphisms. Therefore, $L$ is an isomorphism if and only if $\mathrm{Lef}$ is an isomorphism. The main result of this section has just been established. 
\begin{theorem} \label{thm: main, i guess}
    $(\g,\eta)$ is $1$-Lefschetz if and only if $(\h,\omega)$ is $1$-Lefschetz. 
\end{theorem} 
\begin{remark} \label{obs: main, i guess}
    Let $\Gamma \backslash G$ be a solvmanifold endowed with an invariant contact form $\eta$, and whose corresponding Lie algebra $\g$ has nontrivial center. If further $H^*(\g) \cong H^*(\Gamma \backslash G)$, a condition fulfilled under the conditions of Proposition \ref{prop: sometimes it is an isomorphism}, then Theorem \ref{thm: main, i guess} extends to the geometric setting in a natural way. When the cohomology of $\Gamma \backslash G$ is \textit{not} given by invariant forms, and if further the contact Lie algebra $(\g,\eta)$ fails to be $1$-Lefschetz, then Proposition \ref{prop: there is always an injection} can be used to argue that $\Gamma \backslash G$ is not $1$-Lefschetz.  
\end{remark} 
\begin{remark}
    In the unimodular case, $\omega$ is a non-exact form on $\h$ while its pullback to $\g$ is exact. The proof of Lemma \ref{lemma: 2-cohomologia} below implies that actually $H^2(\g) \cong H^2(\h)/\R\omega$. This gives a clue that the relation between $H^k(\g)$ and $H^k(\h)$ for $2 \leq k \leq n$ is not straightforward.
\end{remark}

\subsection{Some applications} \label{section: some applications}

\indent Except for the last result of this section, we restrict our attention to the nilpotent setting. We first obtain the following generalization of Proposition \ref{prop: for later use}. 
\begin{corollary} \label{cor: 1-lefschetz nilpotent}
    The only contact nilpotent Lie algebras that are $1$-Lefschetz are the Heisenberg Lie algebras.
\end{corollary}
\begin{proof}
    Let $(\g, \eta)$ be a contact nilpotent Lie algebra. Since $\z(\g)$ is nontrivial, and thus generated by $\xi$ according to  Proposition \ref{prop: center in contact Lie algebras}, it arises as contactization of a symplectic Lie algebra $(\h, \omega)$. Notice that $\h$ is nilpotent as observed in Remark \ref{obs: iff unimod, nil, solv}. Moreover, $\h$ is abelian if and only if $\g$ is a Heisenberg Lie algebra, as pointed out in Remark \ref{obs: heisenberg are precisely central extensions of abelian lie algebras}. Therefore, if $(\h,\omega)$ is nonabelian then it fails to be $1$-Lefschetz as a consequence of Benson and Gordon's result, Theorem \ref{thm: Benson-Gordon 1}. But if it is abelian then it is trivially $1$-Lefschetz. The claim then follows from Theorem \ref{thm: main, i guess}.
\end{proof}
\indent The following immediate consequence can be interpreted as the contact counterpart of Benson and Gordon's result on symplectic nilmanifolds, Theorem \ref{thm: Benson-Gordon 1}, in the invariant setting. Recall that a nilmanifold $M = \Gamma \backslash N$ is said to be \textit{Heisenberg} if $N$ is isomorphic to a Heisenberg Lie group, and that any nilmanifold diffeomorphic to a Heisenberg nilmanifold is in fact a Heisenberg nilmanifold due to Remark \ref{obs: heisenberg nilmanifolds}. 
\begin{theorem} \label{thm: odd benson-gordon}
   A nilmanifold endowed with an invariant contact form is $1$-Lefschetz if and only if it is a Heisenberg nilmanifold. 
\end{theorem} 
\indent In particular, the fundamental group of a nilmanifold endowed with an invariant $1$-Lefschetz contact form is isomorphic to a lattice of one of the Heisenberg groups (which have been classified in \cite[Theorem 2.4]{GW}). 
\begin{remark}
    Theorem \ref{thm: odd benson-gordon} provides a partial answer to the question raised in \cite{Cagliari 3} concerning the existence of non-Heisenberg (and thus non-Sasakian, according to Theorem \ref{thm: sasakian nilmanifolds}) contact hard-Lefschetz nilmanifolds, in the negative. Our answer is indeed partial, as there are nilmanifolds admitting contact forms which are not invariant: see Remark \ref{obs: decomposable lie algebras are not contact}. 
\end{remark} 
\indent Combining Remark \ref{obs: rather trivial}, Theorem \ref{thm: sasakian nilmanifolds}, and Theorem \ref{thm: odd benson-gordon}, we obtain the following generalization to the examples found in \cite{Cagliari 1}.  
\begin{corollary} \label{cor: cagliari}
    Any non-Heisenberg nilmanifold endowed with an invariant contact form admits a compatible $K$-contact metric but does not admit a compatible (not necessarily invariant) Sasakian structure.  
\end{corollary}
\indent We use the characterization given in Theorem \ref{thm: Benson-Gordon 2} in combination with Theorem \ref{thm: odd benson-gordon} to refine the description of contact $1$-Lefschetz Lie algebras with nontrivial center. 
\begin{corollary} \label{cor: boring benson-gordon}
    Let $(\g, \eta)$ be a contact unimodular Lie algebra with $\z(\g) = \R \xi$, arising as contactization of a symplectic unimodular Lie algebra $(\h, \omega)$. Denote $\n_{\g} := [\g,\g]_{\g}$ and $[\h, \h]_{\h}$. Assume further that $(\g, \eta)$ is $1$-Lefschetz. Then:
\begin{enumerate} [\rm (i)]
    \item $\n_{\g} = \R \xi \oplus \n_{\h}$ as vector spaces. Moreover, $\n_{\g}$ is the contactization of $\n_{\h}$ with associated $2$-cocycle the restriction of $\omega$ to $\n_{\h}$, with contact form given by restriction of $\eta$ to $\n_{\g}$.
    \item There is an even-dimensional abelian subalgebra $\a$ in $\g$ contained in $\h$ and satisfying $\g = \a \oplus \n_{\g}$ and $\h = \a \oplus \n_{\h}$ as vector spaces. Moreover, the restriction of $\omega$ to $\a$ is a symplectic form on $\a$.  
\end{enumerate} 
\end{corollary}
\begin{proof} Notice that $(\h, \omega)$ is $1$-Lefschetz as per Theorem \ref{thm: odd benson-gordon}. 
\begin{enumerate} [\rm (i)]
    \item The vector space decomposition was already observed in Proposition \ref{prop: xi belongs in the commutator}. The rest of the claim follows from the fact that Theorem \ref{thm: Benson-Gordon 2} ensures that the restriction of $\omega$ to $\n_{\h}$ is a symplectic form on $\n_{\h}$.
    \item The existence of a even-dimensional abelian subalgebra in $\h$ satisfying $\h = \a \oplus \n_{\h}$ as vector spaces and the fact that the restriction of $\omega$ to $\a$ is a symplectic form on $\a$ were already observed in Benson and Gordon's result, Theorem \ref{thm: Benson-Gordon 2}. The fact that $\a$ can be chosen so that $\g = \a \oplus \n_{\g}$ as vector spaces follows from (i).  \qedhere
\end{enumerate}
\end{proof}

\section{Examples of 1-Lefschetz contact solvmanifolds} \label{section: lattices}

    \indent In this section, we construct examples of compact contact solvmanifolds satisfying the $1$-Lefschetz contact condition. Some of these examples are also $2$-Lefschetz, while others are not. The Lie algebras associated with these solvmanifolds are almost nilpotent Lie algebras. Except for the isolated example in Section \ref{section: BG case}, which is the contactization of an already almost nilpotent symplectic Lie algebra appearing in \cite[Example 3]{BG2}, our examples are obtained as contactizations of symplectic almost abelian Lie algebras. We rely on the results of \cite{AG1} and \cite{AG2} concerning the Lefschetz condition for this class of Lie algebras.

    \indent Recall that a Lie algebra $\g$ is called \textit{almost nilpotent} if it has a codimension-one nilpotent ideal. In this case, $\g$ can be written as a semidirect product $\g = \R \ltimes_D \n$ for some $D \in \mathrm{Der}(\n)$. The simply connected Lie group $G$ corresponding to $\g$ can be written as $\R\ltimes_{\phi} N$, where $N$ is the simply connected Lie group corresponding to $\n$ and $\phi \in \mathrm{Aut}(N)$ is obtained by exponentiating $D$ in $N$. Such a group $G$ is called \textit{almost nilpotent} as well. Notice that all almost nilpotent Lie algebras are solvable, and the same is true for almost nilpotent Lie groups. In the particular case where $\n$ is an abelian Lie algebra (and hence $N$ is an abelian group) both $\g$ and $G$ are called \textit{almost abelian}. For almost nilpotent Lie algebras, there is a criterion ensuring the existence of lattices in the corresponding Lie groups.
\begin{proposition} \label{prop: yamada}
    Let $\g = \R \ltimes_D \n$ be a unimodular almost nilpotent Lie algebra. If there is a nonzero $t_0 \in \R$ and a rational basis $\mathcal{B} = \{X_1,\ldots,X_n\}$ of $\n$ for which the matrix of $\exp(t_0 D)$ has integer entries, then the corresponding Lie group $G$ admits lattices. Moreover, at least one such lattice is of the form $\Gamma = t_0 \Z \ltimes_\phi \exp^N(\Span_\Z \{X_1,\ldots,X_n\})$.
\end{proposition} 

    \indent In Proposition \ref{prop: yamada}, $\exp(t_0D)$ is the matrix exponential in $\n \cong \R^{\dim \n}$, which coincides with $d_e(\phi(t_0))$ after a choice of basis. Consequently, there is an equivalent formulation of Proposition \ref{prop: yamada} at the Lie group level; however, we state it in the form most commonly used in practice. Lattices that respect the semidirect product decomposition of $G$, as in Proposition \ref{prop: yamada}, are called \textit{splittable}. 
        
    \indent Proposition \ref{prop: yamada} is a particular case of several more general results, including \cite[Theoréme 3]{Saito2}, \cite[Theorem 3.13]{VGS}, and \cite[Main Theorem]{Yamada}. The formulation given above is closest to that of the latter reference. It is worth pointing out that, although these results provide if-and-only-if criteria for the existence of lattices in terms of properties of $G$, none of them imply that \textit{all} lattices in $G$ are splittable.
    
    \indent For $k\in \N$, $k\geq 3$, the real numbers
\begin{align}\label{eq: tm}
    t_k:=\log \frac{k+\sqrt{k^2-4}}{2}, \quad \text{$k \in \N$ with $k \geq 3$}, 
\end{align} 
    \noindent play a fundamental role in the ensuing constructions. Note that $\alpha := \e^{t_k}$ and $\alpha^{-1} := \e^{-t_k}$ are the only roots of $p_k(x) := x^2 - kx + 1 \in \Z[x]$. From here, it is easy to establish that $\alpha^2 = k \alpha - 1$ and that $\alpha^3 = (k^2-1) \lambda - k$. The next fact is established by induction on $\ell$.  
\begin{lemma}\label{lemma: integers}
    If, for some $k \in \Z$, $\alpha$ and $\alpha^{-1}$ denote the roots of $p_k(x) = x^2-kx+1$ then, for any $\ell \in \N$, $\alpha^{\ell} + \alpha^{-\ell} \in \Z$.
\end{lemma}
    \indent Except for the one in Section \ref{section: BG case}, all examples under consideration arise as contactization of almost abelian Lie algebras. A few remarks are in order. 

    \indent Let $\h_A := \R \ltimes_A \R^{2m+1}$ be an almost abelian Lie algebra. Here $A$ is a matrix thought to be acting on $\R^{2m+1}$. Choose a basis $\{f_1, f_2, x_1, \ldots, x_{2m}\}$ of $\h_A$ where $\ad_{f_1}|_{\R^{2m+1}} = A$ and
\begin{align*}
    \R^{2m+1} = \Span\{f_2, x_1, \ldots, x_{2m}\}, \quad \u_0 := \Span\{x_1, \ldots, x_{2m}\}.
\end{align*}    
    \indent To clarify, we are \textit{defining} $\u_0$. Take $A$ to be
\begin{align} \label{eq: matrixA}
    A = (0) \oplus A_0, \quad A_0 \in \mathfrak{sp}(m,\R)
\end{align}
    \noindent in this basis; in particular, $A$ is traceless, and so $\h_A$ is unimodular. Take a symplectic form $\omega$ on $\h_A$ such that
\begin{align*}
    f^1 \wedge f^2 + \omega_0, \quad  \text{$\omega_0 \in \alt^2 \u_0^*$ is symplectic on $\u_0$}.
\end{align*}
    \noindent Actually, \textit{all} $\h_A$ as above are symplectic, and also \textit{any} symplectic form $\omega$ on $\h_A$ is as described above. Proof for both these claims can be found in \cite[Proposition 5.3, Theorem 5.6]{etal} and \cite[Proposition 4.1 and Remark 4.2]{LW}. While there are more general matrices $A$ for which $\h_A$ is symplectic, none of them is $1$-Lefschetz (due to \cite[Theorem 4.24]{AG2}) and it is unclear whether they have lattices, and so we omit them from the discussion. Often we employ the notation
\begin{align*}
    u_i := x_i \text{ for $1 \leq i \leq m$}, \quad v_j := x_{m+j} \text{ for $1 \leq j \leq m$},
\end{align*}
\noindent with corresponding dual basis denoted by superscripts. We also denote by $J_m(\lambda)$ the elementary Jordan block of size $m \times m$ with eigenvalue $\lambda \in \mathbb{R}$, 
\begin{align} \label{eq: elementary jordan block}
    J_m(\lambda) = 
    \left[
	\begin{array}{c c c c c c}      
		\lambda & 0 & 0 & \cdots & 0 & 0\\ 
        1 & \lambda & 0 & \cdots & 0 & 0\\ 
        0 & 1 & \lambda & \cdots & 0 & 0 \\
        \vdots & \vdots & \vdots & & \vdots & \vdots \\
        0 & 0 & 0 & \cdots & \lambda & 0 \\
        0 & 0 & 0 & \cdots & 1 & \lambda
	\end{array} \right].
\end{align}
\noindent Set $\Gamma := x^1 \wedge \cdots \wedge x^{2m}$, which is the top form in $\u_0$, and define 
\begin{gather*}
    \Gamma_a := x^1 \wedge \cdots \wedge \widehat{x^a} \wedge \cdots \wedge x^{2m} \in \alt^{2m-1} \u_0^*, \\
    \Gamma_{b, c} := x^1 \wedge \cdots \wedge \widehat{x^b} \wedge \cdots \wedge \widehat{x^c} \wedge \cdots \wedge x^{2m}  \in \alt^{2m - 2} \u_0^*, 
\end{gather*}
\noindent where $b < c$. The symbol $\widehat{x^a}$ means that the $1$-form $x^a$ does not appear in the expression for $\Gamma_a$, and similarly for $\Gamma_{b,c}$. Also, set $\delta := f^1 \wedge f^2$. This is a similar notation to that used by the authors in \cite{AG1}. 

\subsection{A 2-Lefschetz example} \label{section: semisimple case}

    \indent Let $m\geq 1$ and $k_1, \ldots, k_m\geq 3$ be integers. Consider the matrix 
\begin{align*}
    A_0 := \diag(t_{k_1}, \ldots, t_{k_m}) \oplus \diag(- t_{k_1}, \ldots, - t_{k_m}) 
\end{align*}
    \noindent where each $t_{k_j}$ is given as in \eqref{eq: tm}, not necessarily all different. Define $A = (0) \oplus A_0$ as in \eqref{eq: matrixA}. As pointed out earlier, we work in the basis $\{u_1,\ldots,u_m,v_1,\ldots, v_m\}$ of $\R^{2m}$. We set $n := m + 1$. For all $2 \leq k \leq n$, set
\begin{gather*}
    \gamma_1 := \delta \in \alt^2 \h_A^*, \quad \gamma_k := u^k \wedge v^k \in \alt^2 \h_A^*, \\
    \overline{\gamma}_1 := \Gamma \in \alt^{2n-2} \h_A^*, \quad \overline{\gamma}_k := \Gamma_{k,k+n} \in \alt^{2n-2} \h_A^*. 
\end{gather*}
    \noindent The Lie algebra $\h_A = \R f_1\ltimes_A \R^{2m+1}$ carries a symplectic form
\begin{align*}
    \omega := \sum_{l=1}^n \gamma_l = \delta + \omega_0,
\end{align*}
    \noindent which is unique up to equivalence (see \cite[Theorem 1.1]{CM}); moreover, $\omega$ is hard-Lefschetz (see \cite[Corollary 1.5]{Kasuya1}). Notice that we are using a notation similar but ultimately different from the one in \cite{AG1}. There is another set of useful $2$-forms and $(2n-2)$-forms, defined for each $2 \leq k \leq n$ as follows: 
\begin{gather*}
    \sigma_1 := \sum_{l=1}^n \gamma_l, \quad \sigma_k := \gamma_k - \gamma_1, \\
    \overline{\sigma}_1 := \sum_{l=1}^n (-1)^{l-1} \overline{\gamma}_l, \quad \overline{\sigma}_k := \overline{\gamma}_k - \overline{\gamma}_1.
\end{gather*}
    \noindent Notice that $\sigma_1 = \omega$. It is straighforward to check, for all $2 \leq k \leq n$, that
\begin{align*}
    \overline{\sigma}_l = \sigma_1 \wedge \dots \wedge \widehat{\sigma}_l \wedge \cdots \wedge \sigma_n;
\end{align*}
    \noindent moreover, all relations are invertible, since
\begin{gather*}
    \gamma_1 = \frac{1}{n} \left(\sigma_1 - \sum_{l=2}^n \sigma_l\right), \quad \gamma_k = \sigma_k + \gamma_1, \\
    \overline{\gamma}_1 = \frac{2}{3 + (-1)^n} \left( \overline{\sigma}_1 + \sum_{l=2}^n (-1)^l \overline{\sigma}_l \right), \quad \overline{\gamma}_k := \overline{\sigma}_k + \overline{\gamma}_1,
\end{gather*}
    \noindent for all $2 \leq k \leq n$. Define
\begin{align*}
    W := \Span\{\sigma_k \mid 2 \leq k \leq n\}, \quad \overline{W} :=  \Span\{ \overline{\sigma}_k \mid 2 \leq k \leq n\}.
\end{align*}
    \noindent Recall from \cite[Lemma 4.1]{AG1} that $\gamma_i \in \alt^2 \h_A^*$ are closed and non-exact for all $1 \leq i,j \leq n$. Whether or not they are the only closed non-exact $2$-forms depends on the choice of the numbers $t_{k_1}, \ldots, t_{k_m}$: for example, when they are taken to be linearly independent over $\Z$, they are the only such forms; and when they are all equal then 
\begin{align} \label{eq: thetas_{i|j}}
    \theta_{i|j} := u^i \wedge v^j \in \alt^2 \u_0^* \text{ for $1 \leq i,j \leq n$ with $i \neq j$}
\end{align}
    \noindent is closed and non-exact for all such $i$, $j$, and not cohomologous to any $\gamma_k$. A similar discussion applies to $(2n-2)$-forms as well. Combining the discussion in \cite[Section 4]{AG1} with our own remarks above, we get that     
\begin{gather*}
    H^2(\h_A) = \Span\{ \gamma_l \mid 1 \leq l \leq n\} \oplus U = \Span \{\sigma_l \mid 1 \leq l \leq n\} \oplus U = \R \sigma_1 \oplus W \oplus U, \\ 
    H^{2n-2}(\h_A) = \Span\{ \overline{\gamma}_l \mid 1 \leq l \leq n \} \oplus \overline{U} = \Span\{ \overline{\sigma}_l \mid 1 \leq l \leq n \} \oplus \overline{U} = \R \overline{\sigma}_1 \oplus \overline{W} \oplus \overline{U}.
\end{gather*}
    \noindent Notice that we are indulging in the slight abuse of language of refering to cohomology classes by suitably chosen representatives of such classes. Here $U$ is the set of $2$-forms as in equation \eqref{eq: thetas_{i|j}} that are closed and non-exact; as observed above, it can have no nonzero $2$-forms (see \cite[Theorem 4.5]{AG1}), all $2$-forms $\theta_{i|j}$ (see \cite[Theorem 4.7]{AG1}), or something in between. Also, $\overline{U}$ is the Poincaré dual of $U$ in $H^{2n-2}(\h_A)$. We do not need an explicit description of neither $U$ nor $\overline{U}$ in what follows.  
    
    \indent Let $(\g_A, \eta)$ denote the contactization of $(\h_A,\omega)$. According to Theorem \ref{thm: main, i guess}, $(\g_A,\eta)$ is $1$-Lefschetz because $(\h_A, \omega)$ is too. It turns out that it is also $2$-Lefschetz. 
\begin{lemma} \label{lemma: 2-cohomologia}
    In the notation described above, 
\begin{align*}
    H^2(\g_A) = W \oplus U, \quad H^{2n-1}(\g_A) = \eta \wedge \overline{W} \oplus \eta \wedge \overline{U}.
\end{align*}
    \noindent In particular, $\dim H^2(\g_A) = \dim H^{2n-1}(\g_A) = n-1$. 
\end{lemma}
\begin{proof}
    According to Lemma \ref{lemma: they are the same differential}, both $d_{\g}$ and $d_{\h}$ coincide on forms pulled back from $\h_A$. Hence, we draw the following two conclusions:
\begin{itemize}
    \item $Z^2(\g_A) = Z^2(\h_A)$: Lemma \ref{lemma: they are the same differential} provides the easy inclusion; for the other one, notice that for any $\mu \in \alt^1 \h_A^*$ we have
\begin{align*}
    d_{\g}(\eta \wedge \mu) = - \omega \wedge \mu - \eta \wedge d_{\h} \mu \in \alt^3 \h_A^* \oplus \eta \wedge \alt^2\h_A^*,
\end{align*}
    \noindent and so $\eta \wedge \mu$ is $d_{\g}$-closed if and only if both $\omega \wedge \mu$ and $\eta \wedge d_{\h} \mu$ are zero separately, and both conditions imply that $\mu = 0$. So all closed $2$-forms on $\g_A$ are precisely the ones pulled back from $\h_A$.
    \item $B^2(\g_A) = B^2(\h_A) \oplus \R \omega$: Notice that all of the $1$-forms on $\g_A$ are either $\eta$ or pulled back from $\h_A$, and also $d_{\g} \eta = - \omega$. The rest follows from Lemma \ref{lemma: they are the same differential}.
\end{itemize}
    \noindent Therefore $H^2(\g_A) \cong H^2(\h_A)/ \R\omega$, and thus $H^2(\g_A) \cong W \oplus U$. Now,
\begin{align*}
    d_{\g} (\eta \wedge \overline{\sigma}_l) = - \omega \wedge \overline{\sigma}_l - \eta \wedge d_{\h} \overline{\sigma}_l = - \sigma_1 \wedge \overline{\sigma}_l = 0 \text{ for all $2 \leq l \leq n$}, 
\end{align*}
    \noindent and so $\{ \eta \wedge \overline{\sigma}_l \mid 2 \leq l \leq n \}$ is a space of closed $(2n-1)$-forms on $\g_A$ of the same dimension as $H^2(\g_A)$. The isomorphism $H^{2n-1}(\g_A) = \eta \wedge \overline{W} \oplus \eta \wedge \overline{U}$ then follows from a direct application of Poincaré duality, as outlined in the proof of Lemma \ref{lemma: H2n-1 is easy for unimodular} (or as in \cite[Theorem 4.14]{AG2}). 
\end{proof}
    \indent Let $L^{n-2}:H^2(\h_A) \to H^{2n-2}(\h_A)$ be the $2$-Lefschetz map on $\h_A$, defined as $L^{n-2}\beta := \omega^{n-2} \wedge \beta$. According to \cite[Corollary 5.6 and Corollary 5.13]{AG1}, $L^{n-2}(U) = \overline{U}$ and 
\begin{align*}
    L^{n-2}\gamma_i = \sum_{j \neq i} \overline{\gamma}_j = \sum_{j = 1}^n \overline{\gamma}_j - \overline{\gamma}_i \text{ for all $1 \leq i \leq n$}.
\end{align*}
    \noindent Hence, for all $2 \leq k \leq n$, it follows that
\begin{gather*}
    L^{n-2}\sigma_1 = \sum_{i=1}^n L^{n-2}\gamma_i = \sum_{i=1}^n \left(  \sum_{j = 1}^n \overline{\gamma}_j - \overline{\gamma}_i \right) = (n-1)  \sum_{l = 1}^n \overline{\gamma}_l, \\
    L^{n-2}\sigma_k = L^{n-2}\gamma_k - L^{n-2}\gamma_1 = \left( \sum_{j = 1}^n \overline{\gamma}_j - \overline{\gamma}_k \right) - \left( \sum_{j = 1}^n \overline{\gamma}_j - \overline{\gamma}_1 \right) = \overline{\gamma}_1 - \overline{\gamma}_k = - \overline{\sigma}_k.
\end{gather*}  
    \noindent Notice in particular that $L^{n-2}(W) = \overline{W}$; and, as observed above, $L^{n-2}(U) = \overline{U}$. Thus, the restriction $L^{n-2} \vert_{W \oplus U} \to L^{n-2}(W \oplus U)$ is an isomorphism. Notice that it is not true that $L^{n-2}$ sends $\sigma_1$ to $\overline{\sigma}_1$, but this is not issue since there is no reason to expect that $L$ respects the vector space decompositions of $H^2(\h_A)$ and $H^{2n-2}(\h_A)$.  
\begin{proposition} \label{prop: semisimple is 2-lefschetz}
    $(\g_A, \eta)$ is $2$-lefschetz.
\end{proposition}
\begin{proof}
    Since $H^2(\g_A) = W \oplus U$ due to Lemma \ref{lemma: 2-cohomologia}, it follows that all clases in $H^2(\g_A)$ have $\xi$-horizontal representatives. Moreover, those representatives can be chosen to be primitive, since
\begin{align*}
    L^{n-1} \sigma_k &= \omega^{n-1} \wedge \sigma_k = \left( \sum_{l=1}^n \gamma_l \right)^{n-1} \wedge (\gamma_k - \gamma_1) \\
    &= (n-1)! \left( \sum_{l=1}^n \overline{\gamma}_l \right) \wedge (\gamma_k - \gamma_1) \\
    &= (n-1)! \left( \overline{\gamma}_k \wedge \gamma_k - \overline{\gamma}_1 \wedge \gamma_1 \right) = 0
\end{align*}
    \noindent for all $2 \leq k \leq n$. That is, $L^{n-1}(W) = 0$. Also,
\begin{align*}
    L^{n-1} \theta_{i|j} = (n-1)! \left( \sum_{l=1}^n \overline{\gamma}_l \right) \wedge (u^i \wedge v^j) = 0
\end{align*}
    \noindent for all $1 \leq i,j \leq n$ such that $i \neq j$, since each term in $\left( \sum_{l=1}^n \overline{\gamma}_l \right)$ is divisible by either $u^i$ or $v^j$ (or both) if $i \neq j$. That is, $L^{n-1}(U) = 0$. This means that the operator $\mathrm{Lef}^2\colon H^2(\g_A) \to H^{2n-1}(\g_A)$ given by $\mathrm{Lef}^2( [\beta] ) := [\eta \wedge L^{n-2} \beta]$ is well defined. Moreover, since $L^{n-2}(W) = \overline{W}$ as remarked above, certainly $\eta \wedge L^{n-2}(W) = \eta \wedge \overline{W}$. Thus, according to Lemma \ref{lemma: 2-cohomologia}, $\mathrm{Lef}^2$ is surjective; since both $H^2(\g_A)$ and $H^{2n-1}(\g_A)$ have the same dimension due to the same Lemma, it follows that $\mathrm{Lef}^2$ is bijective.  
\end{proof}
    \indent Denote by $G_A$ the simply connected Lie group corresponding to $\g_A$.
\begin{proposition} \label{prop: semisimple has lattices}
    For every choice of $k_1, \ldots, k_m$, the solvable Lie group $G_A$ admits lattices. 
\end{proposition}
\begin{proof}
    The proof is an application of Proposition \ref{prop: yamada}. The Lie brackets on $\g_A$ are given by
\begin{align*}
    [f_1,u_i] = t_{k_i} u_i, \quad [f_1,v_i] = - t_{k_i} v_i, \quad  [f_1,f_2] = [u_i, v_i] = \xi, 
\end{align*}
    \noindent for all $1 \leq i \leq m$. Therefore, we may write $\g_A = \R f_1\ltimes_{\tilde{A}} \n$, where
\begin{align*}
    \n := \Span \{\xi, f_2, u_1, v_1, \ldots, u_m, v_m\}
\end{align*}    
    \noindent is $\n$ is the nilradical of $\g_A$, and $\tilde{A}$ is the matrix given in the basis above by
\begin{align*}
    \tilde{A} =
    \begin{bmatrix} 
        0 & 1 \\ 
        0 & 0 
    \end{bmatrix} 
    \oplus 
    \begin{bmatrix} 
        t_{k_1} & \phantom{-} 0\\ 
        0 & -t_{k_1} 
    \end{bmatrix}
    \oplus \cdots \oplus 
    \begin{bmatrix} 
        t_{k_m} & \phantom{-} 0\\ 
        0 & - t_{k_m} 
    \end{bmatrix}.
\end{align*}
    \noindent Thus, 
\begin{align*}
    \exp(\tilde{A}) = 
    \begin{bmatrix} 
        1 & 1 \\ 
        0 & 1 
    \end{bmatrix} 
    \oplus 
    \begin{bmatrix} 
        e^{t_{k_1}} & 0 \\ 
        0 & e^{-t_{k_1}} 
    \end{bmatrix} 
    \oplus \cdots \oplus
    \begin{bmatrix} 
        e^{t_{k_m}} & 0 \\ 
        0 & e^{-t_{k_m}} 
    \end{bmatrix}.
\end{align*}
    \noindent For all $1 \leq i \leq m$, set $\alpha_i = \e^{t_{k_i}}$ and define 
\begin{align*}
    w_i := p_i u_i + q_i \alpha_i v_i, \quad \tilde{w}_i := p_i \alpha_i u_i + q_i v_i,
\end{align*}
    \noindent where $p_i$, $q_i \in \R$ for all $1 \leq i \leq m$. Recall that $\alpha_i^2 = k_i \alpha_i - 1$ and $\alpha_i + \alpha_i^{-1} = k_i$ for all $1 \leq i \leq m$ because of the choice of $t_{k_i}$. Thus,  
\begin{gather*}
    \exp( \tilde{A} )\xi = \xi, \quad \exp(\tilde{A}) f_2 = \xi + f_2, \\
    \exp( \tilde{A} ) w_i =  \tilde{w}_i, \quad \exp( \tilde{A} ) \tilde{w}_i = - w_i + k_i \tilde{w}_i 
\end{gather*}
    \noindent for all $1 \leq i \leq m$, and so the matrix $\exp(\tilde{A})$ in the basis $\{\xi, f_2, w_1, \tilde{w}_1, \ldots, w_m, \tilde{w}_m\}$ has integer coefficients, namely
\begin{align*}
    \begin{bmatrix} 
        1 & 1 \\ 
        0 & 1 
    \end{bmatrix} 
    \oplus 
    \begin{bmatrix} 
        0 & -1\\ 
        1 & \phantom{-} k_1 
    \end{bmatrix} 
    \oplus \cdots \oplus
    \begin{bmatrix} 
        0 & -1 \\ 
        1 & \phantom{-} k_m
    \end{bmatrix}.
\end{align*}    
    \noindent Moreover, if, for all $1 \leq i \leq m$, we set
\begin{align*}
    p_i := 1, \quad q_i := \frac{1}{1-\alpha_i^2},
\end{align*}
    \noindent then, for all $1 \leq i,j \leq m$ with $i \neq j$, we get
\begin{gather*}
    [w_i, \tilde{w}_i] = p_i q_i (1 - \alpha_i^2) \xi = \xi, \\
    [w_i, w_j] = [w_i, \tilde{w}_j] = [w_j, \tilde{w}_i] = [\tilde{w}_i, \tilde{w}_j] = 0 
\end{gather*}
    \noindent Since both $f_2$ and $\xi$ are central in $\n$, this establishes that $\{\xi, f_2, w_1, \tilde{w}_1, \ldots, w_m, \tilde{w}_m\}$ is a rational basis of $\n$. According to Proposition \ref{prop: yamada}, $G_A$ has lattices.  
\end{proof}
    \indent Thus, combining Proposition \ref{prop: semisimple is 2-lefschetz} and Proposition \ref{prop: semisimple has lattices}, one obtains plenty of contact $2$-Lefschetz completely solvable solvmanifolds $\Gamma \backslash G_A$, the contact form on them being the invariant one induced by $\eta$ on $\g_A$. The fact that each $\Gamma \backslash G_A$ is $2$-Lefschetz follows from Proposition \ref{prop: sometimes it is an isomorphism}(ii), since $\g_A$ is completely solvable.  

\subsection{A non 2-Lefschetz example} \label{section: nonsemisimple case}
    \indent Let $k \geq 3$ and $m \geq 2$ be integers. Consider the matrix 
\begin{equation}\label{eq: A0}
    A_0 = J_m(t_k) \oplus J_m(-t_k) \in \mathfrak{sp}(m,\R),
\end{equation}
    \noindent where $t_k$ is given by \eqref{eq: tm}, and $J_m(t_k)$ and $J_m(-t_k)$ as in equation \eqref{eq: elementary jordan block}. Define $A = (0) \oplus A_0$ as in equation \eqref{eq: matrixA}. As pointed out earlier, we work in the basis $\{u_1,\ldots,u_m,v_1,\ldots, v_m\}$ of $\R^{2m}$. The Lie algebra $\h_A=\R f_1\ltimes_A \R^{2m+1}$ carries a symplectic form
\begin{align} \label{eq: forma simplectica loca}
    \omega = \delta + \omega_0, \quad \omega_0 := \sum_{i=1}^m (-1)^{i+1} u^i\wedge v^{m+1-i}
\end{align}    
    \noindent that, furthermore, is $1$-Lefschetz but not $2$-Lefschetz (see \cite[Theorem 4.25]{AG2}). In \cite{AG2}, the $2$-form $\omega_0$ is denoted $g_m(u,v)$, but we have no use for that notation here. According to \cite[Proposition 4.16]{AG2} (as well as the remarks following that result) and \cite[Proposition 4.17]{AG2}, we know that
\begin{align} \label{eq: we known that}
    \omega^{m-1} \wedge (u^1 \wedge v^1) = \pm \delta \wedge \Gamma_{m,2m}, \quad \delta \wedge \Gamma_{m,2m} = - d( f^2 \wedge \Gamma_{m,2m-1} ). 
\end{align}
    \noindent Moreover, \cite[Theorem 4.14]{AG2} ensures that $u^1 \wedge v^1$ represents a nonzero cohomology class on $\h_A$; in said article, it is called $g_1(u,v)$. These facts were used by the authors in \cite{AG2} to show that \textit{no} symplectic form on $\h_A$ is $2$-Lefschetz.  
    
    \indent Let $(\g_A, \eta)$ denote the contactization of $(\h_A,\omega)$. According to Theorem \ref{thm: main, i guess}, $(\g_A,\eta)$ is $1$-Lefschetz because $(\h_A, \omega)$ is too. It turns out that it is not $2$-Lefschetz, as well.

\begin{proposition} \label{prop: nosemisimple is not 2-lefschetz}
    $(\g_A, \eta)$ is not $2$-Lefschetz. 
\end{proposition}
\begin{proof}
    \indent Regard $\alpha := u^1 \wedge v^1 \in \alt^2 \g_A^*$ as a $2$-form on $\g_A$. It is certainly $\xi$-horizontal and $d_{\g}$-closed, this last fact follows from Lemma \ref{lemma: they are the same differential}. Moreover, it is not $d_{\g_A}$-exact since, by the proof of Lemma \ref{lemma: 2-cohomologia}, the only $\xi$-horizontal closed non-exact form on $\h_A$ is exact on $\g_A$ if and only if it is proportional to $\omega$. An alternative, self-contained argument is also possible: if there were some $a \in \R$ and $\beta \in \alt^1 \h_A^*$ such that $\alpha = d_{\g_A} (a \eta +\beta)$ then one would readily arrive at the relation $d_{\h_A} \beta = u^1 \wedge v^1 + a \omega$, implying that $u^1 \wedge v^1 + a \omega \in \alt^2 \h^*_A$ is exact on $\h_A$; this contradicts the description of $H^2(\h_A)$ in \cite[Theorem 4.14]{AG2}. Furthermore, $\alpha$ is primitive: recalling that $n = m +1$ and taking $k = 2$, we see that 
\begin{align*}
    L^{n-k+1} \alpha &= \pm \omega^m \wedge \alpha = \pm \omega \wedge (\omega^{m-1} \wedge \alpha) = \pm \left( \delta + \sum_{i=1}^m (-1)^{i+1} u^i \wedge v^{m+1-i} \right) \wedge \delta \wedge \Gamma_{m,2m} = 0; 
\end{align*}
    \noindent notice that the first relation in equation \eqref{eq: we known that} is used in the third step. It remains to see that $\epsilon_{\eta} L^{n-2}(\alpha) = \eta \wedge \omega^{m-1} \wedge \alpha$ is $d_{\g_A}$-exact. Using the second relation in equation \eqref{eq: we known that}, we see that
\begin{align*}
    - (d_{\g_A} \eta) \wedge f^2 \wedge \Gamma_{m,2m-1} = \omega \wedge f^2 \wedge \Gamma_{m,2m-1} = \left( \delta + \sum_{i=1}^m (-1)^{i+1} u^i \wedge v^{m+1-i} \right) \wedge f^2 \wedge \Gamma_{m,2m-1} = 0,
\end{align*}
\noindent and therefore we obtain 
\begin{align*}
    \eta \wedge \omega^{m-1} \wedge \alpha = \pm \eta \wedge \delta \wedge \Gamma_{m,2m} = - \eta \wedge d_{\g_A}( f^2 \wedge \Gamma_{m,2m-1} ) = \pm d_{\g_A}( \eta \wedge f^2 \wedge \Gamma_{m,2m-1} ).
\end{align*} 
\noindent Hence, $\mathcal{R}_{\mathrm{Lef}_2}$ cannot be the graph of an isomorphism $H^2(\g_A) \to H^{2n-1}(\g_A)$. \qedhere
\end{proof}
    \indent Denote by $G_A$ the simply connected Lie group corresponding to $\g_A$, and denote by $\eta$ the invariant form on $G_A$ induced by the one on $\g_A$. 
\begin{proposition} \label{prop: nonsemisimple has lattices}
    When $m$ is even, the solvable Lie group $G_A$ admits lattices.
\end{proposition}
\begin{proof}
    The proof is an application of Proposition \ref{prop: yamada}. The Lie brackets on $\g_A$ are given by
\begin{align*}
    [f_1,u_i] = J_m(t_k) u_i, \quad [f_1,v_i] = J_m(-t_k) v_i, \quad  [f_1,f_2]=\xi, \quad [u_i,v_{m+1-i}] = (-1)^{i+1}\xi.
\end{align*}
    \noindent for all $1 \leq i \leq m$. Therefore, we may write $\g_A = \R f_1\ltimes_{\tilde{A}} \n$, where
\begin{align*}
    \n := \Span \{\xi, f_2, u_1, v_1, \ldots, u_m, v_m\}
\end{align*}    
    \noindent is the nilradical of $\g_A$, and $\tilde{A}$ is the matrix given in the basis above by
\begin{align*}
    \tilde{A} = 
    \begin{bmatrix} 
        0 & 1 \\ 
        0 & 0
    \end{bmatrix}
    \oplus J_m(t_k) \oplus J_m(-t_k)
    = 
    \begin{bmatrix} 
        0 & 1 \\ 
        0 & 0
    \end{bmatrix}
    \oplus A_0.
\end{align*}
    \noindent Thus, 
\begin{align*}
    \exp(\tilde{A}) = 
    \begin{bmatrix}
        1 & 1 \\ 
        0 & 1 
    \end{bmatrix} 
    \oplus \exp[J_m(t_k)] \oplus \exp[J_m(-t_k)]
    = 
    \begin{bmatrix}
        1 & 1 \\ 
        0 & 1 
    \end{bmatrix} 
    \oplus \exp(A_0).
\end{align*}
    \noindent It is important to mention that $\exp[J_m(\lambda)] = \e^\lambda T_m$ for any $\lambda \in \R$, where $T_m$ is the lower triangular matrix whose value in the $(i,j)$-spot, with $i\geq j$, is $\frac{1}{(i-j)!}$. Note that all the coefficients of $T_m$ are rational numbers. Set $\alpha := \e^{t_k}$ and define
\begin{align} \label{eq: the w's}
    w_1 := u_1 + v_1, \quad w_j := \exp(A_0) w_{j-1} \text{ for all $2 \leq j \leq 2m$},
\end{align}
    \noindent It is a well known result of linear algebra that $w_1$ is a cyclic vector for $\exp(A_0)$. Moreover, since $\exp(A_0) = \alpha T_m\oplus \alpha^{-1} T_m$, its characteristic polynomial is $q(x) := (x^2 - kx + 1)^m \in \Z[x]$. This means that the matrix $\exp(\tilde{A})$ in the basis $\{w_1,\ldots,w_{2m}\}$ has integers coefficients, namely
\begin{align*}
    \begin{bmatrix} 
        1 & 1 \\ 
        0 & 1 
    \end{bmatrix} 
    \oplus C(q), 
\end{align*}    
    \noindent where $C(q)$ is the companion matrix of $q$, which is an integer matrix. 
    
    \indent Let us now verify that the basis $\{\xi, f_2, w_1, \ldots, w_{2m}\}$ is rational. Since the matrix $\exp(A_0) = \alpha T_m\oplus \alpha^{-1} T_m$ and $T_m$ is a rational matrix, there exist $c^r_j\in \Q$ for $1 \leq r \leq m$ and $2 \leq j \leq m$ such that 
\begin{align*}
    w_1=u_1+v_1, \quad \text{and} \quad w_j=\alpha^{j-1}\sum_{r=1}^m c^r_j u_r+\alpha^{-j+1}\sum_{r=1}^m c^r_j v_r, \quad 2\leq j\leq m,
\end{align*}
    \noindent Notice that the same coefficients appear in both sums. Now, for all $2 \leq j \leq 2m$, we compute
\begin{align*}
    [w_1,w_j] & = \left[u_1+v_1,\alpha^{j-1}\sum_{r=1}^m c^r_j u_r+\alpha^{-j+1}\sum_{r=1}^m c^r_j v_r \right] \\
    & =\alpha^{-j+1}c_j^m[u_1,v_m]+\alpha^{j-1}c^m_j[v_1,u_m] \\
    & =c^m_j[\alpha^{-j+1}-(-1)^{m+1}\alpha^{j-1}]\xi \\
    & = c^m_j [\alpha^{-j+1}+\alpha^{j-1}]\xi,
\end{align*}  
    \noindent since $m$ is assumed to be even. Using Lemma \ref{lemma: integers} and the fact that $c^m_j$ is a rational number, we obtain that the structure constants appearing in $[w_1,w_j]$ are rational. Similarly, for all $2 \leq i<j \leq m$, we compute
\begin{align*}
    [w_i,w_j] & = \left[ \alpha^{i-1}\sum_{s=1}^m c^s_i u_s+\alpha^{-i+1}\sum_{s=1}^m c^s_i v_s, \alpha^{j-1}\sum_{r=1}^m c^r_j u_r+\alpha^{-j+1}\sum_{r=1}^m c^r_j v_r \right] \\
    & = \alpha^{i-j}\sum_{s=1}^m c_i^sc_j^{m+1-s} [u_s,v_{m+1-s}]
    +\alpha^{j-i}\sum_{s=1}^m c_i^s c_j^{m+1-s} [v_s,u_{m+1-s}] \\ 
    & =\alpha^{i-j}\sum_{s=1}^m (-1)^{s+1} c_i^sc_j^{m+1-s} \xi 
    -\alpha^{j-i}\sum_{s=1}^m (-1)^{m-s} c_i^s c_j^{m+1-s} \xi \\ 
    & = \sum_{s=1}^m (-1)^{s+1} c_i^s c_j^{m+1-s} [\alpha^{i-j}+\alpha^{j-i}] \xi,
\end{align*}
    \noindent and again we use that $m$ is even in the last equality. Once more, using Lemma \ref{lemma: integers} and the fact that $c^s_j$ is a rational number, we obtain that the structure constants appearing in $[w_i,w_j]$ are rational for all $2 \leq i<j \leq m$. Since both $f_2$ and $\xi$ are central in $\n$, this establishes that $\{\xi, f_2, w_1, \ldots, w_{2m}\}$ is a rational basis of $\n$. According to Proposition \ref{prop: yamada}, $G_A$ has lattices.  
\end{proof} 
\begin{remark}
    As of now, it is unclear whether $G_A$ has lattices for $m$ odd.
\end{remark}
    \indent  Thus, combining Proposition \ref{prop: nosemisimple is not 2-lefschetz} and Proposition \ref{prop: nonsemisimple has lattices}, one obtains plenty of contact $1$-Lefschetz completely solvable solvmanifolds $\Gamma \backslash G_A$ that are not $2$-Lefschetz, the contact form on them being the invariant one induced by $\eta$ on $\g_A$. The fact that each $\Gamma \backslash G_A$ is not $2$-Lefschetz follows from  Proposition \ref{prop: sometimes it is an isomorphism}(ii), since $\g_A$ is completely solvable.    
\begin{remark}
    Analogues to Proposition \ref{prop: nosemisimple is not 2-lefschetz} and Proposition \ref{prop: nonsemisimple has lattices} are valid for the class of almost abelian Lie algebras originating from a matrix
\begin{align*}
    A_0 = A_1 \oplus \cdots \oplus A_r, \quad A_i := J_{m_i}(t_{k_i}) \oplus J_{m_i}(-t_{k_i}) \text{ for all $1 \leq i \leq r$}
\end{align*}
    \noindent where $r \in \N$ is arbitrary, $m_1, \ldots, m_r \in \N$ are all even positive integers, and $t_{k_1}, \ldots, t_{k_r}$ are as in equation \eqref{eq: tm} and such that $k_1, \ldots, k_r$ are all pairwise distinct. The symplectic form $\omega$ on $\h_A$ must be taken to be
\begin{align*}
    \omega := \delta + \omega_1 + \cdots + \omega_r, 
\end{align*}   
    \noindent where each $\omega_j$ is given by a similar formula as that on $\omega_0$ in equation \eqref{eq: forma simplectica loca}. The analogue of Proposition \ref{prop: nosemisimple is not 2-lefschetz} follows as an application to \cite[Lemma 3.5]{AG2}. The analogue of Proposition \ref{prop: nonsemisimple has lattices} follows essentially from the same proof, noting that the fact that $t_{k_i} \neq t_{k_j}$ for all $1 \leq i, j \leq r$ prevents interaction between different blocks, and so the basis given by vectors as in equation \eqref{eq: the w's} works just fine. 
\end{remark} 

\subsection{Another non 2-Lefschetz example} \label{section: BG case}

\indent Let $\h_{\mathrm{BG}}$ be the Lie algebra spanned by
\begin{align*}
    \{w_1, w_2, x_1, y_1, z_1, x_2, y_2, z_2\}
\end{align*}
\noindent whose nontrivial brackets are
\begin{gather*}
    [x_1, y_1] = z_1, \quad [x_2, y_2] = z_2, \\
    [w_1, x_1] = \phantom{+} x_1, \quad [w_1, y_1] = - 2 y_1, \quad [w_1, z_1] = - z_1, \\
    [w_1, x_2] = -  x_2, \quad [w_1, y_2] = \phantom{+} 2 y_2, \quad [w_1, z_2] = \phantom{+} z_2.
\end{gather*}
    \noindent Equivalently, $\h_{BG}$ is the almost-nilpotent Lie algebra $\R w_1 \ltimes_A (\R w_2 \oplus \h_3 \oplus \h_3)$, where each $\h_3$ factor is a $3$-dimensional Heisenberg Lie algebra spanned by $\{x_i, y_i, z_i\}$ for $i = 1$ and $i = 2$, and the action of $A := \ad w_1$ on $\n := \R w_2 \oplus \h_3 \oplus \h_3$ is given by
\begin{align*}
    A = \diag(0, 1, -2, -1, -1, 2, 1). 
\end{align*}
    \noindent Notice that $\n$ is the nilradical of $\h_{\mathrm{BG}}$ and that $\z(\h_{\mathrm{BG}}) = \R w_2$. This Lie algebra appears in \cite[Example 3]{BG2}. A more general family of Lie algebras is studied in \cite{SY}, although they are either isomorphic to $\h_{\mathrm{BG}}$ or have no symplectic forms. Let $\{w^1, w^2, x^1, y^1, z^1, x^2, y^2, z^2\}$ denote the dual basis of $\h_{\mathrm{BG}}^*$. The argument in \cite[Example 3]{BG2}, or direct computation, shows that
\begin{align} \label{eq: description of BG 2-cohomology}
    H^1(\h_{\mathrm{BG}}) = \Span \{w^1, w^2\}, \quad  
    H^2(\h_{\mathrm{BG}}) = \Span \{ w^1 w^2, x^1 z^1, x^2 z^2, x^1 x^2, y^1 y^2\}. 
\end{align}
    \noindent Here $x^1 z^1$ denotes the $2$-form $x^1 \wedge z^1$, and similarly for the rest; the omission of the wedge product is to keep expressions short. Poincaré duality can be used to compute $H^6(\h_{\mathrm{BG}})$ and $H^7(\h_{\mathrm{BG}})$, but we refrain from doing so because we only need the following two relations:
\begin{gather}
    d_{\h_{\mathrm{BG}}}(w^1 w^2 x^1 y^1 z^1) = w^1 w^2 x^1 x^2 y^1 y^2 \in \alt^6 \h_{\mathrm{BG}}^*, \label{eq: thank you, BG} \\
    d_{\h_{\mathrm{BG}}} (w^2 x^1 x^2 y^1 z^1 z^2) = w^1 w^2 x^1 x^2 y^1 z^1 z^2 \in \alt^7 \h_{\mathrm{BG}}^*. \label{eq: thank you, maple}
\end{gather}
    \indent The description of $H^2(\h_{\mathrm{BG}})$ implies that all symplectic forms on $\h_{\mathrm{BG}}$ are cohomologous to
\begin{align} \label{eq: all symplectic forms on BG}
    \omega := a w^1 w^2+ b x^1 z^1 + c x^2 z^2 + e x^1 x^2 + f y^1 y^2
\end{align}
    \noindent for some $a$, $b$, $c$, $e$, $f \in \R$, all nonzero except possibly for $e$. It is straightforward to check, either by direct computation or by appealing to Theorem \ref{thm: Benson-Gordon 2}, that $(\h_{\mathrm{BG}}, \omega)$ is $1$-Lefschetz for all $\omega$. As pointed out in \cite[Example 3]{BG2}, $(\h_{\mathrm{BG}}, \omega)$ is not $2$-Lefschetz for any $\omega$, the reason being that $\rho := x^1 x^2$ belongs in the kernel of the $2$-Lefschetz operator $L^2: H^2(\h_{\mathrm{BG}}) \to H^6(\h_{\mathrm{BG}})$. Indeed, 
\begin{align*}  
    L^2\rho = \omega^2 (x^1 x^2) = 2 a f w^1 w^2 x^1 x^2 y^1 y^2 = d_{\h_{\mathrm{BG}}}(2af w^1 w^2 x^1 y^1 z^1). 
\end{align*}
    \noindent Notice that we have used equation \eqref{eq: thank you, BG} in the last step. 

    \indent Let $(\g_{\mathrm{BG}}, \eta)$ denote the contactization of $(\h_{\mathrm{BG}},\omega)$. The notation hides the fact that $\g_{\mathrm{BG}}$ depends on the choice of parameters $a$, $b$, $c$, $e$, $f \in \R$ in the expression of $\omega$ given in equation \eqref{eq: all symplectic forms on BG}. According to Theorem \ref{thm: main, i guess}, $(\g_{\mathrm{BG}},\eta)$ is $1$-Lefschetz because $(\h_{\mathrm{BG}}, \omega)$ is, too. As in the example in Section \ref{section: nonsemisimple case}, it turns out that it is also not $2$-Lefschetz, and essentially for the same reasons. 

\begin{proposition} \label{prop: BG is not 2-lefschetz}
    $(\g_{\mathrm{BG}}, \eta)$ is not $2$-Lefschetz. 
\end{proposition}
\begin{proof}
    \indent Regard $\rho= x^1 x^2 \in \alt^2 \g_{\mathrm{BG}}^*$ as a $2$-form on $\g_{\mathrm{BG}}$. It is certainly $\xi$-horizontal and $d_{\g_{\mathrm{BG}}}$-closed, this last bit in part because of Lemma \ref{lemma: they are the same differential}. Moreover, it is not $d_{\g_{\mathrm{BG}}}$-exact since, by the proof of Lemma \ref{lemma: 2-cohomologia}, the only $\xi$-horizontal closed non-exact form on $\h_{\mathrm{BG}}$ is exact on $\g_{\mathrm{BG}}$ if and only if it is proportional to $\omega$. An alternative, self-contained argument is also possible: if there were some $k \in \R$ and $\beta \in \alt^1 \h_{\mathrm{BG}}^*$ such that $\rho = d_{\g_{\mathrm{BG}}} (k \eta +\beta)$ then one would readily arrive at the relation $d_{\h_{\mathrm{BG}}} \beta = x^1 x^2 + k \omega$, implying that $x^1 x^2 + k \omega \in \alt^2 \h^*_{\mathrm{BG}}$ is exact on $\h_{\mathrm{BG}}$, which contradicts the description of $H^2(\h_{\mathrm{BG}})$ in equation \eqref{eq: description of BG 2-cohomology}. Furthermore, $\rho$ is primitive: since
\begin{align*}
    \omega^3 = 6 a b c w^1 w^2 x^1 z^1 x^2 z^2 + 6 a b f w^1 w^2 x^1 z^1 y^1 y^2 + 6 a c f w^1 w^2 x^2 z^2 y^1 y^2,
\end{align*}
    \noindent and all terms are divisible either by $x^1$ or by $x^2$, it follows that 
\begin{align*}
    L^3 \rho &= \omega^3 (x^1 x^2) = 0.
\end{align*}
    \noindent It remains to see that $\epsilon_{\eta} L^2 \rho = \eta \omega^2 \rho$ is $d_{\g_{\mathrm{BG}}}$-exact, but this follows from Lemma \ref{lemma: they are the same differential} as well as equations \eqref{eq: thank you, BG} and \eqref{eq: thank you, maple}, since  
\begin{gather*}
    - (d_{\g} \eta) w^1 w^2 x^1 y^1 z^1 = \omega w^1 w^2 x^1 y^1 z^1 = c w^1 w^2 x^1 x^2 y^1 z^1 z^2 = c d_{\g_{\mathrm{BG}}} (w^2 x^1 x^2 y^1 z^1 z^2 ), \\
    \omega^2 (x^1 x^2) = 2 a f w^1 w^2 x^1 x^2 y^1 y^2 = d_{\g_{\mathrm{BG}}}(2af w^1 w^2 x^1 y^1 z^1),
\end{gather*}
\noindent and thus
\begin{align*}
    \eta \omega^2 \rho = \eta d_{\g_{\mathrm{BG}}}(2 af w^1 w^2 x^1 y^1 z^1) = 2 af \, d_{\g_{\mathrm{BG}}} (\eta \, w^1 w^2 x^1 y^1 z^1) - 2acf d_{\g_{\mathrm{BG}}} (w^2 x^1 x^2 y^1 z^1 z^2 ).
\end{align*} 
\noindent Hence, $\mathcal{R}_{\mathrm{Lef}_2}$ cannot be the graph of an isomorphism $H^2(\g_{\mathrm{BG}}) \to H^7(\g_{\mathrm{BG}})$. \qedhere
\end{proof}
    \indent Denote by $G_{\mathrm{BG}}$ and $H_{\mathrm{BG}}$ the simply connected Lie groups corresponding to $\g_{\mathrm{BG}}$ and $\h_{\mathrm{BG}}$ respectively. We now show that $G_{\mathrm{BG}}$ has lattices for some choice of the parameters $a$, $b$, $c$, $e$, $f \in \R$ in equation \eqref{eq: all symplectic forms on BG}. As a byproduct, we reobtain the result in \cite{SY} that $H_{\mathrm{BG}}$ has lattices, but with a different method; see Remark \ref{obs: byproduct} below. 
\begin{proposition} \label{prop: BG has lattices}
    There is a choice of parameters $a$, $b$, $c$, $e$, $f \in \R$ in equation \eqref{eq: all symplectic forms on BG} such that $G_{\mathrm{BG}}$ admits lattices.
\end{proposition}
\begin{proof}
    The proof is an application of Proposition \ref{prop: yamada}. It is best to work with a Lie algebra isomorphic to $\h_{\mathrm{BG}}$, obtained by mapping $w_1 \mapsto t_k w_1$ and leaving all other generators unchanged. Here, $t_k$ is defined as in equation \eqref{eq: tm}, where $k \in \Z$ is an integer satisfying $k \geq 3$. On this isomorphic version of $\h_{\mathrm{BG}}$, choose a symplectic form $\omega$ as given in equation \eqref{eq: all symplectic forms on BG}, and obtain the contactization $(\g_{\mathrm{BG}}, \eta)$ of $(\h_{\mathrm{BG}}, \omega)$, which is isomorphic to the original one. The Lie brackets on $\g_{\mathrm{BG}}$ are given by
\begin{gather*}
    [x_1, y_1] = z_1, \quad [x_2, y_2] = z_2, \\
    [w_1, x_1] = \phantom{+} t_k x_1, \quad [w_1, y_1] = - 2 t_k y_1, \quad [w_1, z_1] = - t_k z_1, \\
    [w_1, x_2] = - t_k x_2, \quad [w_1, y_2] = \phantom{+} 2 t_k y_2, \quad [w_1, z_2] = \phantom{+}t_k z_2, \\
    [x_1, z_1] = b \xi, \quad [x_2, z_2] = c \xi, \\
    [w_1, w_2] = a \xi, \quad [x_1, x_2] = e \xi, \quad [y_1, y_2] = f \xi.
\end{gather*}
    \noindent Recall that all the coefficients $a$, $b$, $c$, $e$, $f \in \R$ are nonzero except possibly for $e$. We make use of this bit of freedom and set $e = 0$. We may write $\g_A = \R w_1\ltimes_{\tilde{A}} \n$, where
\begin{align*}
    \n := \Span \{\xi, w_2, x_1, x_2, y_1, y_2, z_1, z_2\}
\end{align*}    
    \noindent is the nilradical of $\g_A$, and $\tilde{A}$ is the matrix given in the basis above by
\begin{align*}
    \tilde{A} = 
    \begin{bmatrix} 
        0 & a \\ 
        0 & 0
    \end{bmatrix}
    \oplus 
    \begin{bmatrix} 
        t_k & \phantom{-} 0\\ 
        0 & - t_k
    \end{bmatrix}
    \oplus 
    \begin{bmatrix} 
        - 2 t_k & 0\\ 
        \phantom{-}0 & 2 t_k
    \end{bmatrix}
    \oplus 
    \begin{bmatrix} 
        - t_k & 0\\ 
        \phantom{-}0 & t_k
    \end{bmatrix}.
\end{align*}
    \noindent Thus, 
\begin{align*}
    \exp(\tilde{A}) &= 
    \begin{bmatrix}
        1 & a \\ 
        0 & 1
    \end{bmatrix} 
    \oplus
    \begin{bmatrix} 
        e^{t_k} & 0\\ 
        0 & e^{-t_k}
    \end{bmatrix}
    \oplus 
    \begin{bmatrix} 
        e^{-2 t_k} & 0\\ 
        0 & e^{2 t_k}
    \end{bmatrix}
    \oplus 
    \begin{bmatrix} 
        e^{- t_k} & 0\\ 
        0 & e^{t_k}
    \end{bmatrix} \\
    &= \begin{bmatrix}
        1 & a \\ 
        0 & 1 
    \end{bmatrix} 
    \oplus
    \begin{bmatrix} 
        e^{t_k} & 0\\ 
        0 & e^{-t_k}
    \end{bmatrix}
    \oplus 
    \begin{bmatrix} 
        e^{t_k} & 0\\ 
        0 & e^{-t_k}
    \end{bmatrix}^{-2}
    \oplus 
    \begin{bmatrix} 
        e^{t_k} & 0\\ 
        0 & e^{- t_k}
    \end{bmatrix}^{-1} 
\end{align*}
    \indent Set $\alpha := \e^{t_k}$ and define
\begin{gather*} 
    \tilde{x}_1 := \lambda_1 x_1 + \delta_1 \alpha x_2, \quad \tilde{x}_2 := \lambda_1 \alpha x_1 + \delta_1 x_2, \\
    \tilde{y}_1 := \lambda_2 y_1 + \delta_2 \alpha y_2, \quad \tilde{y}_2 := \lambda_2 \alpha y_1 + \delta_2 y_2, \\
    \tilde{z}_1 := \lambda_3 z_1 + \delta_3 \alpha z_2, \quad \tilde{z}_2 := \lambda_3 \alpha z_1 + \delta_3 z_2. \\
\end{gather*}
    \noindent Here, $\lambda_1, \lambda_2, \lambda_3, \delta_1, \delta_2, \delta_3 \in \R$ are nonzero coefficients to be determined explicitly later. For the time being, have in mind that
\begin{align*}
    z_1 = \frac{-1}{\lambda_3 (\alpha^2 -1)} (\tilde{z}_1 - \alpha \tilde{z}_2), \quad z_2 = \frac{-1}{\delta_3 (\alpha^2 -1)} ( \alpha \tilde{z}_1 - \tilde{z}_2).
\end{align*}
    \noindent Recall that $\alpha^2 = k \alpha - 1$ and $\alpha + \alpha^{-1} = k$ because of the choice of $t_k$. Hence, 
\begin{gather*} 
    \exp( \tilde{A} ) \tilde{x}_1 =  \tilde{x}_2, \quad \exp( \tilde{A} ) \tilde{x}_2 = - \tilde{x}_1 + k \tilde{x}_2, \\
    \exp( \tilde{A} ) \tilde{y}_1 =  \tilde{y}_2, \quad \exp( \tilde{A} ) \tilde{y}_2 = - \tilde{y}_1 + k \tilde{y}_2, \\
    \exp( \tilde{A} ) \tilde{z}_1 =  \tilde{z}_2, \quad \exp( \tilde{A} ) \tilde{z}_2 = - \tilde{z}_1 + k \tilde{z}_2. 
\end{gather*}
    \noindent So, the matrix $\exp(\tilde{A})$ in the basis $\{\xi, w_2, \tilde{x}_1, \tilde{x}_2, \tilde{y}_1, \tilde{y}_2, \tilde{z}_1, \tilde{z}_2 \}$ has integers coefficients, namely
\begin{align*}
    \begin{bmatrix} 
        1 & 1 \\ 
        0 & 1 
    \end{bmatrix} 
    \oplus 
    \begin{bmatrix} 
        0 & -1 \\ 
        1 & \phantom{+} k
    \end{bmatrix} 
    \oplus 
    \begin{bmatrix} 
        0 & -1 \\ 
        1 & \phantom{+} k
    \end{bmatrix}^{-2} 
    \oplus 
    \begin{bmatrix} 
        0 & -1 \\ 
        1 & \phantom{+} k
    \end{bmatrix}^{-1} = \begin{bmatrix} 
        1 & 1 \\ 
        0 & 1 
    \end{bmatrix} 
    \oplus 
    \begin{bmatrix} 
        0 & -1 \\ 
        1 & \phantom{+} k
    \end{bmatrix} 
    \oplus 
    \begin{bmatrix} 
        k^2+1 & -k \\ 
        -k & \phantom{+}1
    \end{bmatrix} 
    \oplus 
    \begin{bmatrix} 
        - k & 1 \\
        \phantom{+} 1 & 0
    \end{bmatrix}.
\end{align*}    
    \noindent Now, define
\begin{gather*}
    p := \frac{\lambda_1 \lambda_2}{\lambda_3 (\alpha^2-1)}, \quad q := - \frac{\delta_1 \delta_2}{\delta_3 (\alpha^2-1)}, \\
    r := \lambda_1 \lambda_3 b, \quad s := \delta_1 \delta_3 c, \quad t := \lambda_2 \delta_2 f.
\end{gather*}
    \noindent A direct computation shows that
\begin{gather*}
    [\tilde{x}_1, \tilde{x}_2] = 0, \quad [\tilde{y}_1, \tilde{y}_2] = t (1 - \alpha^2) \xi, \\
    [\tilde{x}_1, \tilde{y}_1] = - (p + q \alpha^3) \tilde{z}_1 + (p \alpha + q \alpha^2) \tilde{z}_2, \\
    [\tilde{x}_1, \tilde{y}_2] = [\tilde{x}_2, \tilde{y}_1] = (p \alpha + q \alpha^2) \tilde{z}_1 + (p \alpha^2 + q \alpha) \tilde{z}_2, \\
    [\tilde{x}_2, \tilde{y}_2] = (p \alpha^2 + q \alpha) \tilde{z}_1 + (p \alpha^3 + q) \tilde{z}_2, \\
    [\tilde{x}_1, \tilde{z}_1] = (r + s \alpha^2) \xi, \quad [\tilde{x}_2, \tilde{z}_2] = (r \alpha^2 + s) \xi, \\
    [\tilde{x}_1, \tilde{z}_2] = [\tilde{x}_2, \tilde{z}_1] = (r + s )\alpha \xi. 
\end{gather*}
    \noindent Have in mind that $[\tilde{x}_1, \tilde{x}_2] = 0$ because we have set $e = 0$; similarly, all other Lie brackets are zero. Thus, a sufficient condition for $\{\xi, w_2, \tilde{x}_1, \tilde{x}_2, \tilde{y}_1, \tilde{y}_2, \tilde{z}_1, \tilde{z}_2 \}$ to be a rational basis of $\n$ is     
\begin{gather}
    p + q \alpha^3 \in \Q, \quad p \alpha + q \alpha^2 \in \Q, \quad    p \alpha^2 + q \alpha \in \Q, \quad p \alpha^3 + q \in \Q, \label{eq: first four} \\
    r + s \alpha^2 \in \Q, \quad (r + s) \alpha \in \Q, \quad r \alpha^2 + s \in \Q, \label{eq: second three} \\
    t (1 - \alpha^2) \in \Q.\label{eq: last one}
\end{gather}  
    \noindent Let's focus on the four conditions in equation \eqref{eq: first four}. We take $p$, $q$ to be in the subfield $\Q(\alpha)$ of $\R$, and argue that there are solutions there. This means to write $p = p_1 + p_2 \alpha$ and $q = q_1 + q_2 \alpha$ for some $p_1$, $p_2$, $q_1$, $q_2 \in \Q$. Recall that $\alpha^2 = k \alpha - 1$ and $\alpha^3 = (k^2-1) \alpha - k$. Putting all together,
\begin{align*}
    p + q \alpha^3 & = -(-p_1 + kq_1 + (k^2-1) q_2) + ( p_2 + (k^2-1)q_1 + k(k^2-2) q_2 ) \alpha, \\
    p \alpha + q \alpha^2 & = - (p_2+q_1+kq_2) + (p_1 + k (p_2 + q_1) + (k^2-1)q_2) \alpha, \\
    p \alpha^2 + q \alpha & = - (p_1 + q_2 + kp_2) + (q_1 + k (q_2 + p_1) + (k^2-1)p_2) \alpha, \\
    p \alpha^3 + q & = - ( k p_1 + (k^2-1) p_2 - q_1) + ( (k^2-1) p_1 + k (k^2-2) p_2 + q_2 ) \alpha. 
\end{align*}
    \noindent So, the four conditions in equation \eqref{eq: first four} are satisfied in $\Q(\alpha)$ if and only if
\begin{gather*}
    \begin{cases}
        p_2 + (k^2-1)q_1 + k(k^2-2) q_2  = 0, \\
        p_1 + k (p_2 + q_1) + (k^2-1)q_2 = 0, \\
        q_1 + k(p_1 + q_2) + (k^2-1)p_2 = 0, \\
        (k^2-1) p_1 + k (k^2-2) p_2 + q_2 = 0, 
    \end{cases} 
\end{gather*}
    \noindent for some $p_1$, $p_2$, $q_1$, $q_2 \in \Q$. The matrix of this system is
\begin{align*}
    M = 
    \begin{bmatrix}
        0 & 1 & k^2 -1 & k(k^2-2) \\
        1 & k & k & k^2-1 \\
        k & k^2-1 & 1 & k \\
        k^2-1 & k(k^2-2) & 0 & 1
    \end{bmatrix}, 
\end{align*}
    \noindent and can be readily shown to have rank $2$. Thus, nontrivial solutions exist. Moreover, all solutions are parametrized by $(u,v) \in \Q^2$ as
\begin{align*}
    p_1 = -\frac{k(k^2-2)}{k^2 -1} u - \frac{1}{k^2 -1}v, \quad p_2 = u, \quad q_1 = - \frac{1}{k^2 -1} u - \frac{k(k^2-2)}{k^2 -1} v, \quad q_2 = v. 
\end{align*}
    \noindent Therefore, the most general solution of equation \eqref{eq: first four} with $p$, $q \in \Q(\alpha)$ is given by
\begin{align*}
    p = -\left(\frac{k(k^2-2)}{k^2 -1} u + \frac{1}{k^2 -1}v\right) + u \alpha, \quad q = -\left(\frac{1}{k^2 -1} u +\frac{k(k^2-2)}{k^2 -1} v\right) + v \alpha. 
\end{align*}
    \noindent In particular, if $(u,v) = (0,1-k^2)$ then
\begin{align*}
    p = 1, \quad q = k(k^2-2) - (k^2-1)\alpha. 
\end{align*}
    \noindent Let's turn to the three conditions in equation \eqref{eq: second three}. We rewrite them using that $\alpha^2 = k \alpha - 1$ as
\begin{align*}
    (r - s) + k s \alpha = l, \quad (r+s) \alpha = m, \quad - (r-s) + k r \alpha = n
\end{align*}
    \noindent for some $l$, $m$, $n \in \Q$. Notice that if we add the first and last equations we get
\begin{align*}
    l + n = (r+s) k \alpha = k m \quad \Longrightarrow \quad r+s = \frac{m}{\alpha} = \frac{l+n}{k \alpha},
\end{align*}
    \noindent and if we substract them we get
\begin{align*}
    l - n = (r-s) (2 - k \alpha) \quad \Longrightarrow \quad r - s = \frac{l - n}{2 - k \alpha}. 
\end{align*}
    \noindent Thus, the general solution of equation \eqref{eq: second three} are parametrized by $(u,w) \in \Q^2$ as
\begin{align*}
    r = \frac{1}{2} \left( \frac{l+n}{k \alpha} + \frac{l - n}{2 - k \alpha} \right), \quad s = \frac{1}{2} \left( \frac{l+n}{k \alpha} - \frac{l-n}{2 - k \alpha} \right)
\end{align*}
    \noindent In particular, if $(l,n) = (k,k)$ then
\begin{align*}
    r = s = \frac{1}{\alpha}
\end{align*}
    \noindent Set $\ell := k(k^2-2) - (k^2-1)\alpha$, which is nonzero since $\alpha$ is irrational. If we choose 
\begin{gather}
    \lambda_1 = 1, \quad \lambda_2 = \alpha^2 - 1, \quad \lambda_3 = 1, \notag \\
    \delta_1 = 1, \quad \delta_2 = - (\alpha^2 - 1) \ell, \quad \delta_3 = 1, \notag \\ 
    b = \frac{1}{\alpha}, \quad c = \frac{1}{\alpha}, \quad f = -\frac{1}{( \alpha^2-1)^3 \ell}, \label{eq: b, c, f}
\end{gather}
    \noindent then we see that all conditions in equations \eqref{eq: first four}, \eqref{eq: second three}, and \eqref{eq: last one} are satisfied. Therefore, for this choice of parameters, $\{\xi, w_2, \tilde{x}_1, \tilde{x}_2, \tilde{y}_1, \tilde{y}_2, \tilde{z}_1, \tilde{z}_2\}$ is a rational basis of $\n$. According to Proposition \ref{prop: yamada}, $G_A$ has lattices.  
\end{proof}
\begin{remark}
    The choice of parameters $a$, $b$, $c$, $e$, $f \in \R$ in the proof of Proposition \ref{prop: BG has lattices} is not unique. Therein, $a$ is left unconstrained, $e$ is taken to be zero, and $b$, $c$, $f$ are determined after $k \in \Z$ as in equation \eqref{eq: b, c, f}. 
\end{remark}
\begin{remark} \label{obs: byproduct} 
     The proof of Proposition \ref{prop: BG has lattices} contains also a proof that $H_{\mathrm{BG}}$, the simply connected Lie group corresponding to $\h_{\mathrm{BG}}$, admits lattices. Moreover, our proof is different than the one in \cite{SY}. The argument is roughly the same, having to forget to account for $\xi$ (or take all parameters $a$, $b$, $c$, $e$, $f$ to be zero, so the extension is trivial). The same change of basis in $\n$ works; in particular, we arrive at the same conditions in equation \eqref{eq: first four} (but not to the conditions in equations \eqref{eq: second three} and \eqref{eq: last one}). 
\end{remark}
    \indent Thus, combining Proposition \ref{prop: BG is not 2-lefschetz} and Proposition \ref{prop: BG has lattices}, one obtains a contact $1$-Lefschetz completely solvable solvmanifold $\Gamma \backslash G_{\mathrm{BG}}$ that is not $2$-Lefschetz, the contact form on it being the invariant one induced by $\eta$ on $\g_{\mathrm{BG}}$. The fact that $\Gamma \backslash G_{\mathrm{BG}}$ is not $2$-Lefschetz follows from  Proposition \ref{prop: sometimes it is an isomorphism}(ii), since $\g_A$ is completely solvable.

\ 
\end{document}